\documentclass[11pt]{article}

\usepackage{hyperref}
\usepackage{a4wide}
\usepackage{graphicx}
\usepackage{amsmath}
\usepackage{amssymb}
\usepackage{amsthm}
\usepackage{enumerate}

\usepackage{combelow}
\usepackage{mathdots}

\newtheorem{theorem}{Theorem}[section]
\newtheorem{definition}[theorem]{Definition}

\newtheorem{lemma}[theorem]{Lemma}
\newtheorem{claim}[theorem]{Claim}

\newtheorem{conjecture}[theorem]{Conjecture}

\newtheorem{problem}[theorem]{Problem}
\newtheorem{observation}[theorem]{Observation}
\newtheorem{exercise}[theorem]{Exercise}

\newcommand{\lam}{{\lambda}}
\newcommand*{\symdiff}{\bigtriangleup}
\newcommand{\DGM}{D_{G,M}}
\newcommand{\MAX}{\mathrm{max}}
\newcommand{\twr}{\mathrm{twr}}

\begin{document}

\title{An approximate version of a conjecture of Aharoni and Berger}

\author{{Alexey Pokrovskiy \thanks{Department of Mathematics, ETH, 8092 Zurich, Switzerland. {\tt dr.alexey.pokrovskiy@gmail.com}. Research supported in part by SNSF grant SNSF grant 200021-175573 and the Methods for Discrete Structures, Berlin graduate school (GRK 1408).}}\\ \\
 ETH Z\"urich.}
%\\
%\\ \small Keywords: Rainbows.}

\maketitle

\begin{abstract}
Aharoni and Berger conjectured that in every proper edge-colouring of a bipartite multigraph by $n$ colours with at least $n+1$ edges of each colour there is a rainbow matching using every colour. This conjecture generalizes a longstanding problem of Brualdi and Stein about transversals in Latin squares.
Here an approximate version of the Aharoni-Berger Conjecture is proved---it is shown that if there are at least $n+o(n)$ edges of each colour in a proper $n$-edge-colouring of a bipartite multigraph then there is a rainbow matching using every colour.
\end{abstract}
\vspace{1cm}
\tableofcontents
\newpage
\section{Introduction}
%In this paper we study rainbow matchings in properly edge-coloured bipartite graphs. A graph is properly edge-coloured if its edges are coloured so that no two vertices at an edge have the same colour. A rainbow matching in a graph is a set of edges sharing no vertices, and having different colours. It is natural to look for large rainbow subgraphs of properly coloured graphs. Aharoni and Berger made the following conjecture about rainbow matchings in bipartite multigraphs.
%\begin{conjecture}\label{AharoniBergerConjecture}
% Let $G$  be a properly edge-coloured bipartite multigraph with $n$ colours and at least $n+1$ edges of each colour. Then $G$ has a a rainbow matching using every colour. 
%\end{conjecture}
%The above conjecture is motivated by old problems about transversals in Latin squares. Recall that a Latin square of order $n$ is an $n\times n$ array filled with $n$ different symbols, where no symbol appears in the same row or column more than once. A partial transversal in a Latin square is a set of entries such that no two entries are in the same row, same column, or have the same symbol. A transversal in a Latin square is a partial transversal which uses very symbol.  There is a one-to-one correspondence betee
The research in this  paper is motivated by some old problems about transversals in Latin squares.
Recall that a Latin square of order $n$ is an $n\times n$ array filled with $n$ different symbols, where no symbol appears in the same row or column more than once. 
%Latin squares arise in different branches of mathematics such as algebra (where Latin squares are exactly the multiplication tables of quasigroups) and  experimental design (where they give rise to designs called Latin square designs). They also occur in  recreational mathematics---for example completed Sudoku puzzles are Latin squares. %
 A transversal in a Latin square of order $n$ is a set of $n$ entries such that no two entries are in the same row, same column, or have the same symbol.
% One reason transversals in Latin squares are interesting is that a Latin square has an orthogonal mate if, and only if, it has a decomposition into disjoint transversals. See \cite{WanlessSurvey} for a survey about transversals in Latin squares.
It is easy to see that not every Latin square has a transversal (for example the unique $2\times 2$ Latin square has no transversal.) However, it is possible that every Latin square contains a large \emph{partial transversal}. Here, a partial transversal of size $m$ means a  set of $m$ entries such that no two entries are in the same row, same column, or have the same symbol. The study of transversals in Latin squares goes back to Euler who  studied orthogonal Latin squares i.e. order $n$ Latin squares  which can be decomposed into $n$ disjoint transversals. For a survey of transversals in Latin squares, see~\cite{WanlessSurvey}.

There are several closely related, old, and difficult conjectures which say that Latin squares should have large partial transversals. The first of these is a conjecture of Ryser that every Latin square of odd order contains a transversal \cite{Ryser}.  Brualdi  conjectured that every Latin square contains a partial transversal of size $n-1$ (see \cite{Brualdi}.) Stein independently made the stronger conjecture that every $n\times n$ array filled with $n$ symbols, each appearing exactly $n$ times contains a partial transversal of size $n-1$ \cite{Stein}. Because of the similarity of the above two conjectures, the following is often referred to as ``the Brualdi-Stein Conjecture''.
\begin{conjecture}[Brualdi and Stein,~\cite{Brualdi, Stein}]\label{ConjectureBrualdiStein}
 Every $n\times n$ Latin square has a partial transversal of size $n-1$.
\end{conjecture}

In this paper we will study a generalization of the Brualdi-Stein Conjecture to the setting of rainbow matchings in properly coloured bipartite multigraphs. How are these related? There is a one-to-one correspondence between $n\times n$ Latin squares and proper edge colourings of $K_{n,n}$ with $n$ colours. 
Indeed consider a Latin square $S$ whose set of symbols is $\{1, \dots, n\}$ with the $i,j$ symbol $S_{i,j}$. To $S$ we  associate 
 an edge-colouring of $K_{n,n}$ with the colours $\{1, \dots, n\}$, by setting $V(K_{n,n})=\{x_1, \dots, x_n, y_1, \dots, y_n\}$ and letting the edge between $x_i$  and  $y_j$ receive colour $S_{i,j}$. Notice that this colouring is proper i.e. adjacent edges receive different colours.
Recall that a \emph{matching} in a graph is a set of disjoint edges. We call a matching \emph{rainbow} if all of its edges have different colours. It is easy to see that partial transversals in the Latin square $S$ correspond to  rainbow matchings in the corresponding coloured $K_{n,n}$. Thus the Brualdi-Stein Conjecture is equivalent to the statement that ``in any proper $n$-edge-colouring of $K_{n,n}$, there is a rainbow matching of size $n-1$.'' 
Once the conjecture is phrased in this form, one begins to wonder whether large rainbow matchings should exist in more general coloured graphs.% i.e. perhaps the Brualdi-Stein Conjecture is a special case of some more general phenomenon. 
Aharoni and Berger made the following generalization of the Brualdi-Stein Conjecture.
\begin{conjecture}[Aharoni and Berger,~\cite{AharoniBerger}]\label{AharoniBergerConjecture}
 Let $G$  be a properly edge-coloured bipartite multigraph with $n$ colours having at least $n+1$ edges of each colour. Then $G$ has a  rainbow matching using every colour. 
\end{conjecture}
This conjecture attracted a lot of attention since it was made. 
A most natural way of attacking it is to consider graphs which have substantially more than $n+1$ edges in each colour, and show that such graphs have a rainbow matching using every colour. For example an easy greedy argument shows that every properly edge-coloured bipartite multigraph with $n$ colours and at least $2n$ edges of each colour has a rainbow matching of size $n$. Indeed, if the largest matching $M$ in such a graph had size $\leq n-1$, then one of the $2n$ edges of the unused colour would be disjoint from $M$, and we could get a larger matching by adding it.
This simple bound has been successively improved by many authors.
Aharoni, Charbit, and Howard \cite{AharoniCharbitHoward} proved that matchings of size $\lfloor 7n/4\rfloor$ are sufficient to guarantee a rainbow matching of size $n$. Kotlar and Ziv \cite{KotlarZiv} improved this to $\lfloor 5n/3\rfloor$. 
The author proved that $\phi n+o(n)$ is sufficient, where $\phi\approx 1.618$ is the Golden Ratio~\cite{PokrovskiyRainbow1}.
 Clemens and Ehrenm\"uller \cite{DennisJulia}  showed that $3n/2+o(n)$ is sufficient. 
The best currently known bound is by   Aharoni, Kotlar, and Ziv \cite{aharoni2017representation} who showed that having $3n/2+1$ edges of each colour in an $n$-edge-coloured bipartite multigraph guarantees a rainbow matching of size  $n$.

Additionally, there are two results showing that just $(1+o(1))n$ edges in each colour are enough if we place additional assumptions on $G$. A special case of a theorem of Haggkvist and Johansson~\cite{HaggkvistJohansson} (proved by probabilistic methods) is that ``every bipartite graph consisting of $n$ edge-disjoint perfect matchings of size $n+o(n)$ edges has a rainbow matchings of size $n$''. The author showed that the assumption that the matchings are perfect can be removed i.e. every bipartite graph consisting of $n$ edge-disjoint matchings  of size  $n+o(n)$ edges has a rainbow matching  of size $n$~\cite{PokrovskiyRainbow1}.

The goal of this paper is to improve on all previous asymptotic results by showing that $(1+o(1))n$ edges are sufficient for all bipartite multigraphs.
\begin{theorem}\label{MainTheorem}
For all $\epsilon >0$, there exists an $N_0=N_0(\epsilon )$ such that the following holds.
Let $G$ be a properly coloured bipartite multigraph with $n\geq N_0$ colours and  at least $(1+\epsilon )n$ edges of each colour. Then $G$ contains a rainbow matching using every colour.
\end{theorem}
The above theorem is the natural approximate version of Conjecture~\ref{AharoniBergerConjecture}. Now the interesting direction for further research is to try and improve the second order term.

This theorem is proved by associating an auxiliary directed graph with $G$ and studying certain kinds of paths in the directed graph. Such an approach was also taken in the author's previous paper~\cite{PokrovskiyRainbow1}, and is substantially refined here. In the next section we give an overview of the various components of the proof of Theorem~\ref{MainTheorem}.

\section{Proof sketch}
The proof of Theorem~\ref{MainTheorem} is quite long and complicated. 
The basic idea is to associate an auxiliary directed graph to $G$ and then study properties of this directed graph. 
The directed graph is studied by introducing five new concepts---``switching paths'', ``amidstness'', ``reaching'', ``bypassing'', and ``$\lam$-components''---and then proving many lemmas about these concepts. Since these concepts are quite foreign, we use this section to give a slow and detailed introduction to all of them. In particular we motivate some of these concepts by showing how they relate to the initial undirected graph in Theorem~\ref{MainTheorem}.

This section and the main proof of Theorem~\ref{MainTheorem} (Sections~\ref{SectionBipartite} and~\ref{SectionDirectedGraphs}) can be read completely independently of one another. All concepts that we introduce in this section, will be reintroduced during the main proof of Theorem~\ref{MainTheorem} (usually more concisely.)

\subsection{Associating a directed graph}
Let $G$ be a properly coloured bipartite multigraph as in Theorem~\ref{MainTheorem}, and let $M$ be a rainbow matching of maximum size in $G$. Suppose for the sake of contradiction that $M$ doesn't use every colour.
Aside from~\cite{HaggkvistJohansson}, all approaches to Conjecture~\ref{AharoniBergerConjecture} have involved performing local manipulations on $M$ to try and produce a larger rainbow matching. Here a ``local manipulation'' on $M$ means choosing some edge $m\in M$ and $e\not\in M$ such that $M-m+e$ is another rainbow matching of the same size as $M$. The basic idea of the proof is to perform a sequence of such local manipulations to obtain a new matching $M'$ of the same size as $M$ such that there is some edge which can be added to $M'$ to give a larger rainbow matching. Since $M$ was originally chosen to have maximum size, this gives a contradiction.

Thus the main aim throughout the proof is to find a suitable sequence of local manipulations. A key idea in~\cite{PokrovskiyRainbow1} was that such sequences correspond to paths in a suitable auxiliary directed graph. The following is the directed graph which we will use.
\begin{definition}[The directed graph $\DGM$]
Let $G$ have bipartition classes $X$ and $Y$, $C_G$ the set of colours in $G$, and $C_M$ the set of colours on $M$. Let $X_0=X\setminus V(M)$. For any colour $c\in C_M$, let $m_c$ be the colour $c$ edge of $M$.
The digraph $\DGM$ corresponding to $G$ and $M$ is defined as follows:
\begin{itemize}
\item The vertex set of $\DGM$ is the set $C_G$ of colours of edges in $G$.  
\item For two colours $u$ and $v\in V(\DGM)$ there is a directed edge from $u$ to $v$ in $\DGM$ whenever there is a colour $u$ edge from some $x\in X$ to the  vertex $m_v\cap Y$. 
\end{itemize}
\end{definition}

\begin{figure}[ht]
  \centering
    \includegraphics[width=0.8\textwidth]{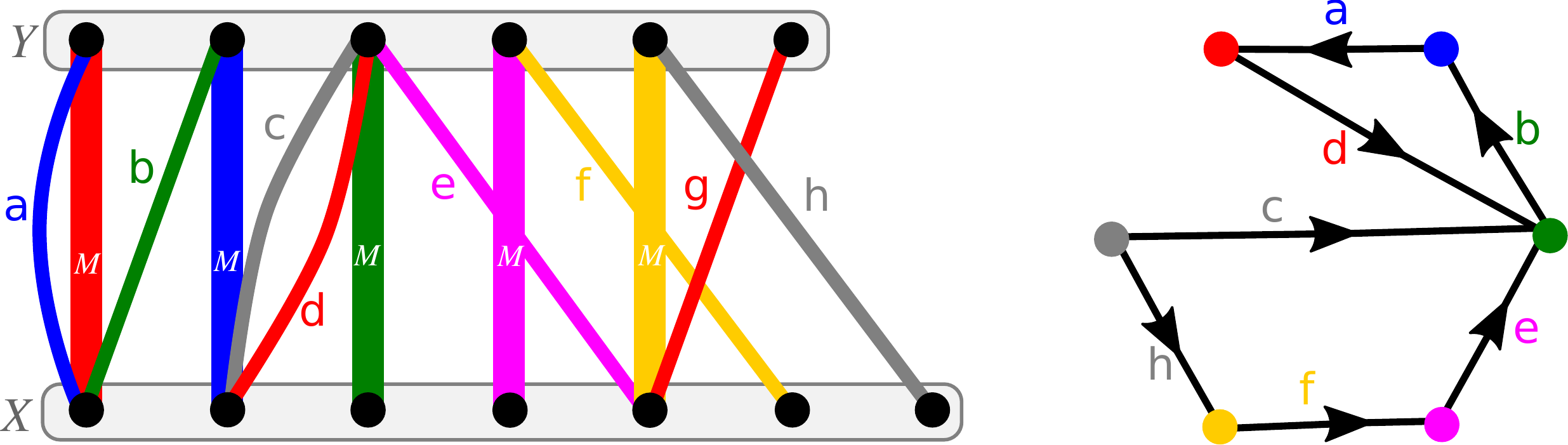}
  \caption{A graph $G$, with a matching $M$, and the corresponding directed  graph $\DGM$. The thick vertical edges labelled ``$M$" are the rainbow matching $M$. All other edges are denoted by $a$ -- $h$ to show which edge of $\DGM$ corresponds to which edge of $G$. Notice that the edge $g$ of $G$ doesn't have a corresponding edge in $\DGM$---this is because $g$ doesn't go through $Y\cap V(M)$.}\label{FigureDGMunlabelled}
\end{figure}

See Figure~\ref{FigureDGMunlabelled} for a diagram of a bipartite multigraph and the corresponding directed graph $\DGM$.
Consider the directed path  in the $\DGM$  with edge sequence $(h,f,e)$ and vertex sequence (grey, yellow, pink, green). Notice that deleting the yellow, pink, and green edges from $M$ and replacing them with $h$, $f$, and $e$ produces a new rainbow matching of the same size as $M$. In addition this new matching misses a different colour (green rather than grey.) This demonstrates that directed paths in $\DGM$ can give the kinds of local manipulations we are interested in.

However not all directed paths in $\DGM$ correspond to sequences of local manipulations. For example in Figure~\ref{FigureDGMunlabelled}, the directed path $c,b,a$ doesn't work since the three edges $c,b,a$ in $G$ do not form a matching. In fact it is easy to check that the only directed paths in Figure~\ref{FigureDGMunlabelled} which correspond to the kinds of manipulations we're interested in are sub-paths of $(h,f,e)$.

The previous paragraphs show that while paths in $\DGM$ can capture the kind of local manipulations we're looking for, not all paths do so. 
We will add labels to the edges of $\DGM$ in order to be able to describe exactly the kind of paths we're interested in. The set of labels  for edges of $\DGM$ is $X_0\cup C_M$ (where $X_0=X\setminus V(M)$ and $C_M$ is the set of colours of the edges of $M$.) If there is a colour $u$ edge in $G$ from $x\in X$ to $m_v\cap Y$, then we label the corresponding edge $uv\in \DGM$ by the following rule.
\begin{itemize}
\item If $x\in X_0$ then the edge $uv$ is labelled by $x$.
\item If $x\in m_c\in M$ then $uv$ is labelled by $c$, the colour of $m_c$.
\end{itemize}
\begin{figure}[ht]
  \centering
    \includegraphics[width=0.8\textwidth]{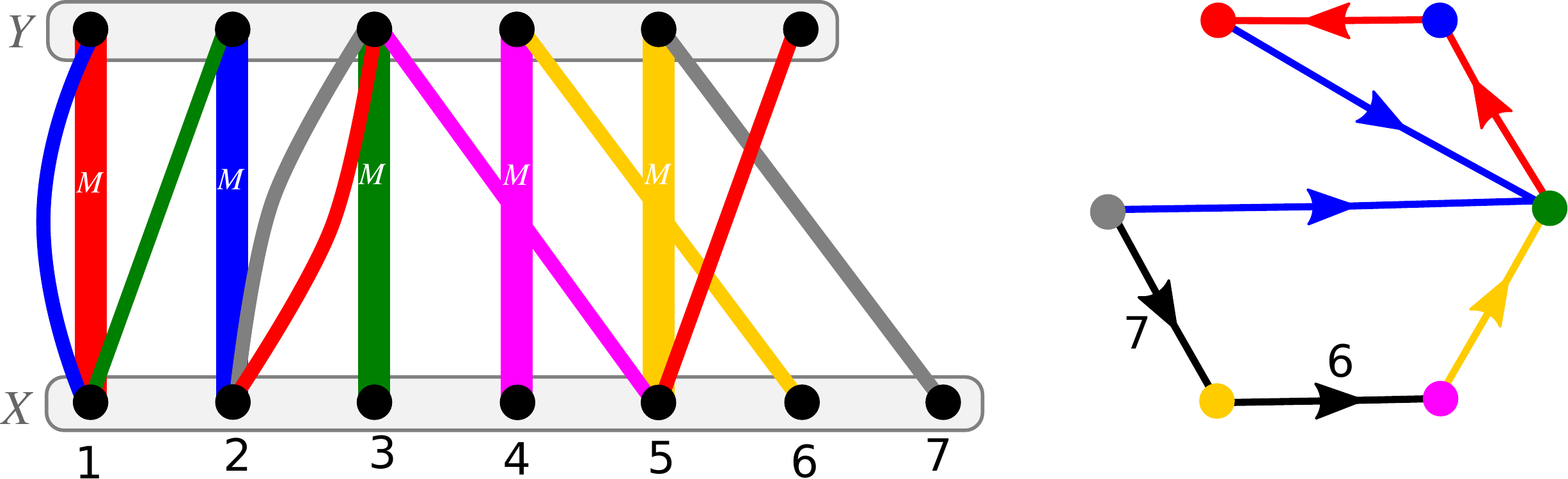}
  \caption{The same graphs $G$ and $\DGM$ as in Figure~\ref{FigureDGMunlabelled}, but now with the edge-labels on $\DGM$. The coloured edges in $\DGM$ are ones labelled by elements of $C_M$ (or equivalently the ones labelled by something in $V(\DGM)$.) The black edges are ones labelled by elements of $X_0$ (or equivalently ones labelled by something not in $V(\DGM)$.)}\label{FigureDGMlabelled}
\end{figure}
See Figure~\ref{FigureDGMlabelled} for an example of this labelling. 
One key point to notice is that the set of labels $X_0\cup C_M$ is not just an ambient set---since $V(\DGM)=C_G$ an element of $C_M$ can simultaneously be a vertex of $\DGM$ and a label of edges in $\DGM$. 
Formally, an edge-labelled  directed graph is defined to be a directed graph $D$ together with a set $X_0$ with $X_0\cap V(D)=\emptyset$ and a labelling function $f:E(D)\to V(D)\cup X_0$.
The set $X_0$ is called the set of \emph{non-vertex labels} in $D$. We call $X_0\cup V(D)$ the \emph{set of labels} in $D$ (regardless of whether $D$ actually has edges labelled by all elements of $X_0\cup V(D)$).

Having equipped $\DGM$ with a labelling, we can define the kinds of paths we are interested in.
\begin{definition}[Switching path]\label{Definition_Switching_Path_ProofSketch}
A path $P=(p_0, \dots, p_d)$ in an edge-labelled, directed graph $D$ is a switching path if the following hold.
\begin{itemize}
\item $P$ is rainbow  i.e. the edges of $P$ have different labels.
\item If $p_{i}p_{i+1}$ is labelled by a vertex $v\in V(D)$, then $v=p_j$ for some $1\leq j\leq i$.
\end{itemize}
\end{definition}
\begin{figure}[ht]
  \centering
    \includegraphics[width=0.8\textwidth]{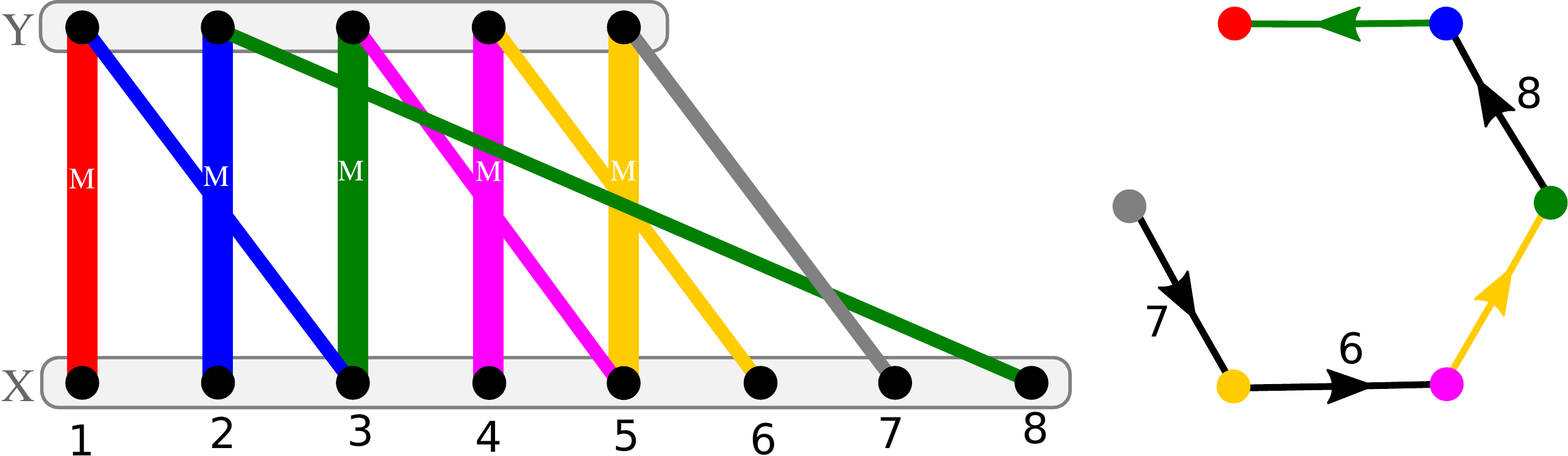}
  \caption{A switching path in a graph $\DGM$. Replacing the edges of $M$ for the other pictured edges of $G$ produces a new matching of the same size as $M$.}\label{FigureSwitchingPath}
\end{figure}
In other words a switching path is a rainbow path with a kind of ``consistency'' property for its edge-labels which are vertices: For every edge $e\in P$ which is labelled by a vertex $v$, $P$ must pass through $v$ before it reaches $e$. Notice that this vertex $v$ is not allowed to be $p_0$, the starting vertex of $P$. A consequence of this is that the first edge $p_0p_1$ of $P$ cannot be labelled by a vertex of $D$ (in the case of $\DGM$ this means that the first edge of any switching path must be labelled by something in $X_0$).

See Figure~\ref{FigureSwitchingPath} for an example of a switching path.
Notice that this path does correspond to the kinds of local manipulations of $M$ which we are interested in i.e if we exchange the edges of $M$ for the edges in $G$ corresponding to the switching path, then we obtain a new rainbow matching of the same size as $M$. 
%Also the good path $h,f,e$ we identified in Figure~\ref{FigureDGMunlabelled} is also a switching path (and in fact the only switching paths in Figure~\ref{FigureDGMunlabelled} are subpaths of $h,f,e$.)

When looking at a switching path in the graph is $\DGM$, the vertices of $P$ correspond to edges of $G$ which we want to remove from the matching $M$, and the edges of $P$ correspond to edges of $G$ which we want to add to $M$. The two conditions in the definition of ``switching path'' then have natural interpretations when one seeks to obtain a new rainbow matching by switching the edges along $P$. Asking for the switching path to be rainbow is equivalent asking for the edges we want to add to $M$ not intersecting in $X$ (which is needed to get a matching). The second part of  Definition~\ref{Definition_Switching_Path_ProofSketch} ensures that when we add an edge to $M$, its colour was previously removed from $M$.

The following exercise makes precise how to modify a matching $M$ using a switching path in $\DGM$ starting from a colour outside $M$.
\begin{exercise}\label{ExerciseSwitchingPathExchangeEdges}
Let $M$ be a rainbow matching in a graph $G$, $p_0$ a colour not in $M$, and $P=(p_0,p_1,\dots,p_d)$ a switching path in $\DGM$.
For $i\geq 1$, let $m_i$ be the colour $p_i$ edge of $M$, and for $i\geq 0$, let $e_i$ be the edge of $G$ corresponding to $p_{i}p_{i+1}$. Show that the following is  a rainbow matching missing the colour $p_d$:
$$M+e_0-m_1+e_1 \dots-m_{d-1}+e_{d-1} - m_d.$$
\end{exercise}
For a solution to the above exercise, see Claim~\ref{ClaimSwitchingPathExchangeEdges}. Exercise~\ref{ExerciseSwitchingPathExchangeEdges} is exactly what we use to try and extend $M$ into a larger matching. If $M$ was chosen to be maximum, then Exercise~\ref{ExerciseSwitchingPathExchangeEdges} can be used to show that $\DGM$ possesses a certain degree property. This and other properties of $\DGM$ will be discussed in the next section.

\subsection{Properties of the directed graph}
The labelled directed graph $\DGM$ ends up having several properties which we use in the proof of Theorem~\ref{MainTheorem}. In this section we go through the properties which we need.
See Figure~\ref{FigureDGMColouring} for  examples of some of the features that $\DGM$ can have.
\begin{figure}[ht]
  \centering
    \includegraphics[width=0.8\textwidth]{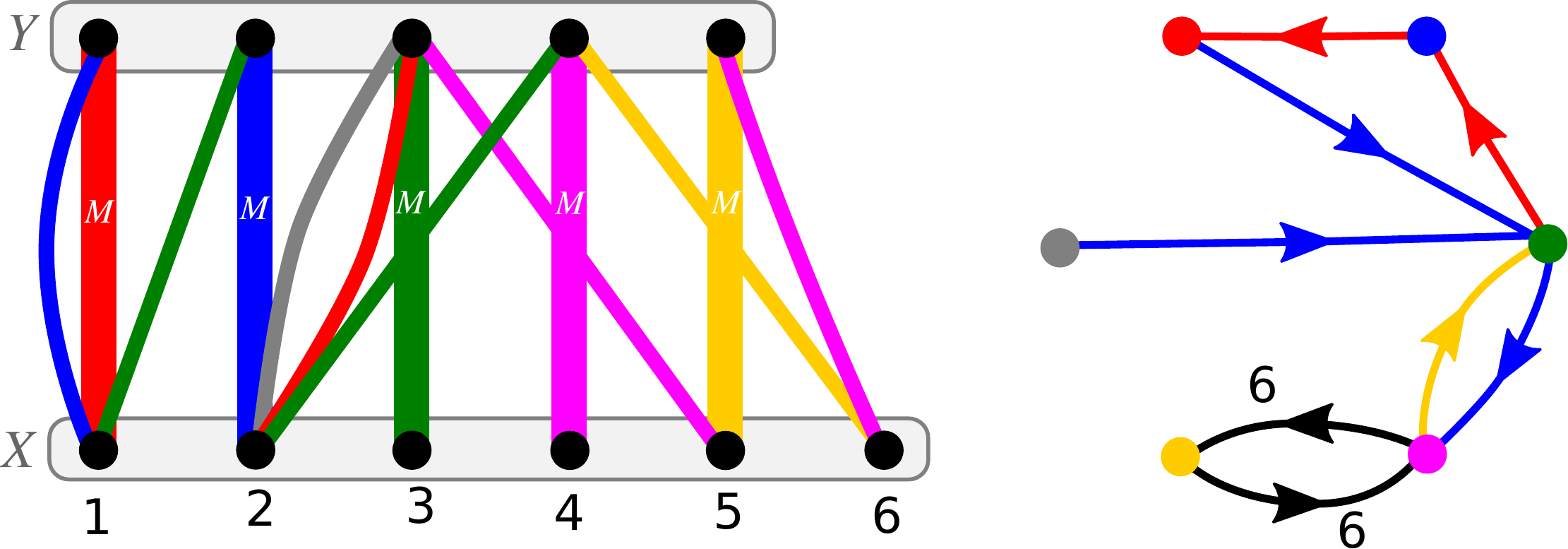}
  \caption{Some of the features $\DGM$ has. The directed graph $\DGM$ doesn't have multiple edges, unless they go in different directions (like the two edges labelled $6$). At a vertex $v$, $\DGM$ never has out-going edges with the same label, but it may have in-going edges with the same label (For example the green vertex has two in-going blue edges.) }\label{FigureDGMColouring}
\end{figure}

For two vertices $u,v\in DGM$ it is possible for $uv$ and $vu$ to both be present in $\DGM$. For example the two edges between the green and pink vertices in Figure~\ref{FigureDGMColouring}. However it is impossible for the edge $uv$ to appear twice with different labels i.e. the directed graph $\DGM$ is simple.
\begin{exercise}\label{ExerciseSimple}
Using the fact that $G$ is properly coloured, show that for $u,v \in V(\DGM)$, there is at most one edge from $u$ to $v$ in $\DGM$.
\end{exercise}
For a solution to this exercise, see Lemma~\ref{LemmaProperlyColoured}.
The labelling on the directed graph $\DGM$ is far from a general labelling. We make the following definitions which generalize proper colouring to directed graphs.
\begin{definition}
Let $D$ be a labelled directed graph.
\begin{itemize}
\item $D$ is out-properly labelled if for any $u\in V(D)$, all out-going edges $uv$ have different labels.
\item $D$ is in-properly labelled if for any $u\in V(D)$, all in-going edges $vu$ have different labels.
\end{itemize}
\end{definition}

It turns out that the labelling on $\DGM$ is always out-proper.
\begin{exercise}\label{ExerciseOutProperlylabelled}
Using the fact that $G$ is properly coloured, show that $\DGM$ is out-properly labelled.
\end{exercise}
For a solution to this exercise, see Lemma~\ref{LemmaProperlyColoured}.
The labelling on $\DGM$ is not always in-proper. For example, in  Figure~\ref{FigureDGMColouring}, the green vertex has two in-going blue edges.
 Notice that in Figure~\ref{FigureDGMColouring} this happened because of the multiple edge in $G$. It turns out that this is the only way to have in-going edges with the same label in $\DGM$.
\begin{exercise}
Suppose that $G$ is properly coloured, simple, and $M$ is a matching in $G$.  Show that $\DGM$ is in-properly labelled.
\end{exercise}

Recall that the special case of Theorem~\ref{MainTheorem} when $G$ is simple was proved in the author's earlier paper~\cite{PokrovskiyRainbow1}. The  case when $G$ is simple turns out to be much easier to prove precisely because the directed graph $\DGM$ associated to $G$ is both in-properly and out-properly labelled. The reason for the difficulty of the multigraph case is that dense directed graphs which are not in-properly labelled do not necessarily have certain connectivity properties. This difficulty is explained in more detail in Section~\ref{SectionProofSketchConnectedness}.

The other main property  of $\DGM$ which we will need is a degree property i.e. we will want to know that all vertices in $\DGM$ have a suitably large degree.
Let $Y_0=Y\setminus S$ be the set of vertices in $Y$ disjoint from the matching $M$.
From the definition of $\DGM$, notice that every edge $e\in G$ corresponds to an edge of $\DGM$ unless $e$ passes through $Y_0$ or $e \in M$ \footnote{The edges of $M$ could be naturally thought of corresponding to loops in $\DGM$, but to keep our analysis to loopless graphs, we won't do this.}. 
A consequence of this is that $e(\DGM)=e(G)-|M|-e(X,Y_0)$. 
Recall that every colour $c$ in $G$ has $(1+\epsilon)n$ edges. For a colour $c$, let $c_{Y_0}$ be the number of colour $c$ edges going through $Y_0$.
From the definition of $\DGM$ we have 
\begin{align}
|N^+(c)|&=(1+\epsilon)n-c_{Y_0}-1 \text{ if $M$ has a colour $c$ edge}, \label{EqIntroExpansion1}\\  
|N^+(c)|&=(1+\epsilon)n-c_{Y_0} \text{  if $M$ has no colour $c$ edge} \label{EqIntroExpansion2}.
\end{align}
Here $N^+(c)$ denotes the out-neighbourhood of $c$ i.e. the set of $x\in V(\DGM)$ with $cx\in E(\DGM)$.
Notice that (\ref{EqIntroExpansion1}) and~(\ref{EqIntroExpansion2}) do not by themselves imply that $|N^+(c)|$ is large for any colour $c$. It is possible that most of the edges of $G$ go through $Y_0$, making the $c_{Y_0}$ term dominant in (\ref{EqIntroExpansion1}) and~(\ref{EqIntroExpansion2}).
However the fact that $M$ is a \emph{maximum size} rainbow matching does force some colours in $G$ to have a large out-degree in $\DGM$. In particular if $c_0$ is a colour which does not appear on $M$, then notice that there cannot be any edges in $G$ between $X_0$ and $Y_0$---indeed if such an edge existed then it could be added to $M$ to give a rainbow matching larger than $M$. 
Recall that from the assumption of Theorem~\ref{MainTheorem} there are $\geq (1+\epsilon)n$ colour $c_0$ edges in $G$, and at most $|M|\leq n$ of these can intersect $X\cap V(M)$.
The other $\epsilon n$ colour $c_0$ edges must go between $X_0$ and $Y\cap V(M)$, giving $|N^+(c_0)|\geq \epsilon n$. 

The above discussion shows that all colours not on $M$ have a high out-degree in $\DGM$. 
Can we get something similar for the other colours in $G$? Recall from Exercise~\ref{ExerciseSwitchingPathExchangeEdges} that switching paths can be used to give new rainbow matchings with the same size as $M$. Using this it is easy to show that any colour close to  $c_0\not\in M$ in $\DGM$ has a large degree in $\DGM$ as well.
\begin{exercise}\label{ExerciseHighX0Degree}
Let $c_0$ be a colour not on $M$, and $c$ some other colour. Let $P$ be a switching path from $c_0$ to $c$ in $\DGM$. Then $|N^+(c)|\geq \epsilon n -|P|$.
\end{exercise}
The above exercise is a special case of Lemma~\ref{LemmaMaximumMatchingExpansion} which we prove later. So far we have looked at only edges labelled by $X_0$ and found that vertices close to missing colours have many such edges leaving them. For a set of labels $L$, define $N^+_{L}(v)$ to be the set of $x\in N(v)$ with $vx$ labelled by some $\ell \in L$. Under the assumptions of Exercise~\ref{ExerciseHighX0Degree}, it is easy to show that $|N^+_{X_0}(c)|\geq \epsilon n -|P|$.

We would like to have information about how big $N^+_L$ is for sets of labels $L$ other than $X_0$. Where could we get such information? In Figure~\ref{FigureAmidst}, notice that if $M$ is a maximum matching, then there cannot be any red edges going from $\{2,4,9\}$ to $Y_0$. Indeed if there was such an edge $e$ then we could look at the rainbow matching $M'$ as in Exercise~\ref{ExerciseSwitchingPathExchangeEdges} (corresponding to the switching path in Figure~\ref{FigureAmidst}) and then add $e$ to $M'$ to get a larger rainbow matching. Thus if there are red edges in $G$ touching $\{2,4,9\}$, then they must go through $Y\cap V(M)$, and hence must have corresponding edges in $\DGM$. Since these edges go  from $v$ to $N^+(v)$, this would tell us that $N^+(v)$ is slightly bigger than the estimate we have in Exercise~\ref{ExerciseHighX0Degree}. For just the single path $P$ in Figure~\ref{FigureAmidst}, this increase is very small. But if we had a large collection of switching paths $P$ like the one in Figure~\ref{FigureAmidst}, then the gains may add up to give a large improvement on the bound in Exercise~\ref{ExerciseHighX0Degree}. The next definition captures what kind of information about the path $P$ in Figure~\ref{FigureAmidst} we were interested in.
\begin{definition}[Amidst]\label{DefinitionAmidstIntroduction}
Let $u$ and $v$ be two vertices in an edge-labelled, directed graph $D$, and $\ell$ a label. We say that $\ell$ is amidst $u$ and $v$ if there is a switching path $P=(u,p_1,\dots, p_{d},v)$ from $u$ to $v$ such that the following hold.
\begin{itemize}
\item There are no edges of $P$ labelled by $\ell$.
\item If $\ell$ is a vertex of $D$ then  $\ell\in \{p_1, \dots, p_d, v\}$.
\end{itemize}
\end{definition}
\begin{figure}[ht]
  \centering
    \includegraphics[width=0.8\textwidth]{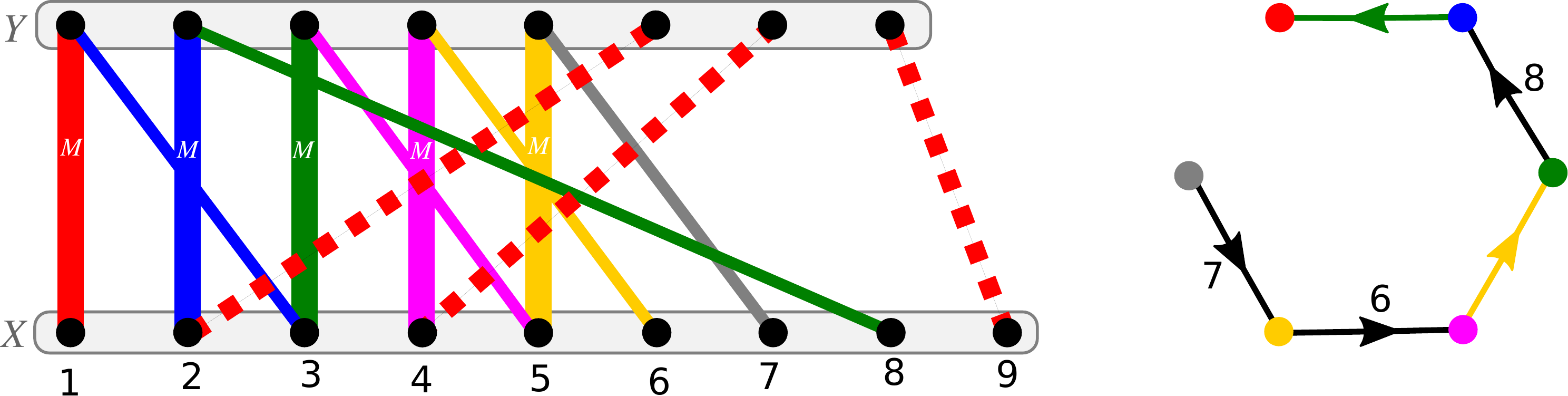}
  \caption{A switching path $P$ in a graph $\DGM$, and three edges that can be added to get a larger matching. 
  Notice that the three  labels $\{\text{blue, pink}, 9\}$ in $\DGM$ are amidst grey and red, as witnessed by the switching path $P$. The vertices in $X$ of blue and pink are $2$ and $4$---which are the $X$-vertices of the corresponding dashed red edges. This shows how amidstness is used to identify vertices of $X$ through which we can add edges to extend $M$. This is the essence of Exercise~\ref{ExerciseAmidstnessEdgesY0}.\newline
Notice that in the above diagram, labels which are not \emph{amidst} a pair of labels cannot be used for augmenting the matching $M$. For example, if there was a red edge $f$ from vertex $5$ to $Y_0$, then one might hope to switch some edges to free up the red colour and vertex $5$ in order to extend the matching by adding $f$. However this cannot be done because freeing up vertex $5$ and colour red would require yellow to be amidst grey and red. In the above diagram yellow is not amidst grey and red. (intuitively because in order to free up red, the yellow edge starting at $6$ must be used).
  }\label{FigureAmidst}
\end{figure}
Notice a parallel between each of the two parts of the definitions of ``switching path'' and ``amidst'': The first parts are about forbidding edges of a path from having particular labels, whereas the second parts are about paths passing through a particular vertex.
This similarity is no coincidence --- a path $P=(p_0, \dots, p_d)$ is a switching path if, and only if, the path $p_0, \dots, p_{d-1}$ is a switching path witnessing the label of $p_{d-1}p_d$ being amidst $p_0$ and $p_{d-1}$.

If $P$ is a path as in Definition~\ref{DefinitionAmidstIntroduction}, then we say that $P$ \emph{witnesses} $\ell$ being amidst $u$ and $v$.
As an example, the path $P$ in Figure~\ref{FigureAmidst} witnesses each of the labels  $\{\text{blue, pink}, 9\}$ being amidst grey and red.
Suppose that $\ell\in C_M$ is the colour of some edge $m$ in $M$.
By an argument similar to the one in the previous paragraph, it is possible to show that if $\ell$ is a label amidst $u$ and $v$, and $u$ is not present on $M$, then there is no colour $v$ edge from $m\cap X$ to $Y_0$. 
\begin{exercise}\label{ExerciseAmidstnessEdgesY0}
Let $\ell, u, v$ be colours in $G$ with $u$ not in $M$ and $\ell$ the colour of an edge $m\in M$.
If $\ell$ is amidst $u$ and $v$, then there is no colour $v$ edge from $m\cap X$ to $Y_0$ in $G$.
\end{exercise}
For a solution to the above exercise see Lemma~\ref{LemmaSwitchingPathVertexAmidst}. 
The essence of the solution is in Figure~\ref{FigureAmidst} --- the dashed red edges are exactly the kind of edges that Exercise~\ref{ExerciseAmidstnessEdgesY0} is about. If any of them were present in the graph then they could be augmented to the matching. 
We now have that given a set of vertices $X'\subseteq X$, if all the corresponding labels are amidst $u$ and $v$, then all the colour $v$ edges touching $X'$ in $G$ must contribute to $N^+(v)$ in $\DGM$.  The following exercise is a strengthening of Exercise~\ref{ExerciseHighX0Degree} which takes into account vertices in $X$ outside $X_0$.
\begin{exercise}\label{ExerciseHighDegreeGeneral}
Suppose that $M$ misses a colour $c^*$, $v$ is a colour in $G$, and $A$ is a set of labels in $\DGM$ which are amidst $c^*$ and $v$. Then $|N^+_A(v)|\geq |A|-|X_0|+\epsilon n-1$.
\end{exercise}
For a solution to this exercise, see Lemma~\ref{LemmaMaximumMatchingExpansion}.
As remarked before, this is actually a strengthening of Exercise~\ref{ExerciseHighX0Degree}. Indeed given a path $P$ as in Exercise~\ref{ExerciseHighX0Degree}, notice that if $x\in X_0$ is a label which does not occur  on edges of $P$, then $x$ is amidst $u$ and $v$ (witnessed by the path $P$.) Applying Exercise~\ref{ExerciseHighDegreeGeneral} with $A$ the set of labels in $X_0$ and not on $P$ we get $|N^+_{A}(v)|\geq |A|-|X_0|+\epsilon n-1\geq \epsilon n-|P|$.

Exercise~\ref{ExerciseHighDegreeGeneral} allows us to finally state the method we use to prove Theorem~\ref{MainTheorem}. We prove that for any $\epsilon>0$, there cannot be arbitrarily large labelled digraphs satisfying the degree condition of Exercise~\ref{ExerciseHighDegreeGeneral}. The following is an intermediate theorem we prove, which implies Theorem~\ref{MainTheorem}.
\begin{theorem}\label{TheoremDigraphIntroduction}
For all $\epsilon$ with $0<\epsilon\leq 0.9$, there is an $N_0=N_0(\epsilon)$ such that the following holds. Let $D$ be any out-properly edge-labelled, simple, directed graph on $n\geq N_0$ vertices. Let $X_0$ be the set of labels which are not vertices of $D$

Then for all $u\in V(D)$, there is a vertex $v$ and a set of labels $A$ amidst $u$ and $v$, such that $|N^+_A(v)|<|A|-|X_0|+\epsilon  n$.
\end{theorem}
We remark that the set $A$ can be an arbitrary subset of $V(D)\cup X_0$ and that $D$ might not have edges labelled by all elements of $A$.

Modulo the discussion in this section, it is easy to see that this theorem implies Theorem~\ref{MainTheorem}. Indeed suppose that there was a sufficiently large graph $G$ as in Theorem~\ref{MainTheorem}. Suppose that a maximum matching $M$ in $G$ doesn't use every colour.  By Exercises~\ref{ExerciseSimple} and~\ref{ExerciseOutProperlylabelled} we know that the corresponding digraph $\DGM$ is out-properly labelled and simple. Let $c^*$ be some colour outside $M$. By Exercise~\ref{ExerciseHighDegreeGeneral} we know that for any $v\in V(\DGM)$, we have $|N^+_{A}(v)|\geq |A|-|X_0|+\epsilon n-1\geq|A|-|X_0|+0.9\epsilon n$ for any set of labels $A$ amidst $c^*$ and $v$. But this contradicts Theorem~\ref{TheoremDigraphIntroduction}.

We conclude this section by explaining how amidstness can be used to build switching paths.
Recall that a path $P=(p_0, \dots, p_d)$ is a switching path if, and only if, the path $p_0, \dots, p_{d-1}$ is a switching path witnessing the label of $p_{d-1}p_d$ being amidst $p_0$ and $p_{d-1}$.
Because of this, labels which are amidst two vertices $u$ and $v$ have potential to be be used to extend switching paths. 
The following exercise makes this precise.
\begin{exercise}\label{ExerciseAmidstAddOneVertex}
If a label $\ell$  is amidst $x$ and $y$ and there is some vertex $z$ such that the edge $yz$ is present and labelled by $\ell$, then there is a switching path from $x$ to $z$. 
\end{exercise}
A version of this exercise is proved in Lemma~\ref{LemmaAmidstAddOneVertex}.
Exercise~\ref{ExerciseAmidstAddOneVertex} is important because it is one of the tools we will use to build longer and longer switching paths.
%informally a label $c$ is amidst $u$ and $v$ if it has the possibility of being used in a switching path starting with $u$ and ending with an edge $vw$ labelled by $c$. Thus the concept of ``amidstness'' ends up being useful for finding  switching paths.

\subsection{The right notion of connectedness}\label{SectionProofSketchConnectedness}
Theorem~\ref{TheoremDigraphIntroduction} is proved studying connectivity properties of subgraphs of $D$. It is not immediately apparent why connectivity is useful here. One  hint of it being useful comes from the definition of ``amidst''. 
The first part of the definition of ``amidst'' asks for a $u$ to $v$ path avoiding all edges of label $\ell$. If there are $< k$ colour $\ell$ edges then this is a property $k$-edge-connected graphs have.
The second part of the definition of ``amidst'' asks for a path going from $u$ to $v$ via some other vertex $\ell$. This is a property which $2$-vertex-connected undirected graphs have (as a consequence of Menger's Theorem).

The purpose of connectivity in the proof is to find sets of vertices $C\subseteq V(\DGM)$ which are highly connected in the following sense---for any pair $u,v\in C$ we have $c$ amidst $u$ and $v$ for most $c\in C$. We can then plug $C$ into the assumption of Theorem~\ref{TheoremDigraphIntroduction} in order to deduce that $v$ has a high out-degree.
Knowing that vertices in $C$ have high out-degree is then used to find a set $C'$  which is also highly connected and substantially larger than $C$. Iterating this process we get larger and larger highly connected sets, which can eventually be used to get a contradiction to these sets being smaller than $V(D)$. 

What notion of connectivity should we use? In \cite{PokrovskiyRainbow1}, the following notion was used.
\begin{definition}\label{DefinitionWeakConnectedness}
 Let $W$ be a set of vertices in a labelled digraph $D$. We say that $W$ is $(k,d)$-{rainbow connected} in $D$ if, for any set of at most $k$ labels $S$ and any vertices $x,y \in W$, there is a rainbow $x$ to $y$ path of length $\leq d$ in $D$  avoiding colours in $S$. 
\end{definition}
This kind of connectivity is useful when the graph $G$ is a simple graph rather than a multigraph. Recall that if $G$ is a simple graph then the labelling on $\DGM$ is both in-proper and out-proper. In~\cite{PokrovskiyRainbow1} it is proved that in any  labelled digraph $D$ there is a highly $(k,d)$-connected set $C$ with  $|C|\geq \delta^+(G) -o(n)$ which is a key intermediate result in proving Theorem~\ref{MainTheorem} in the case when $G$ is simple. 

\begin{figure}[ht]
  \centering
    \includegraphics[width=0.37\textwidth]{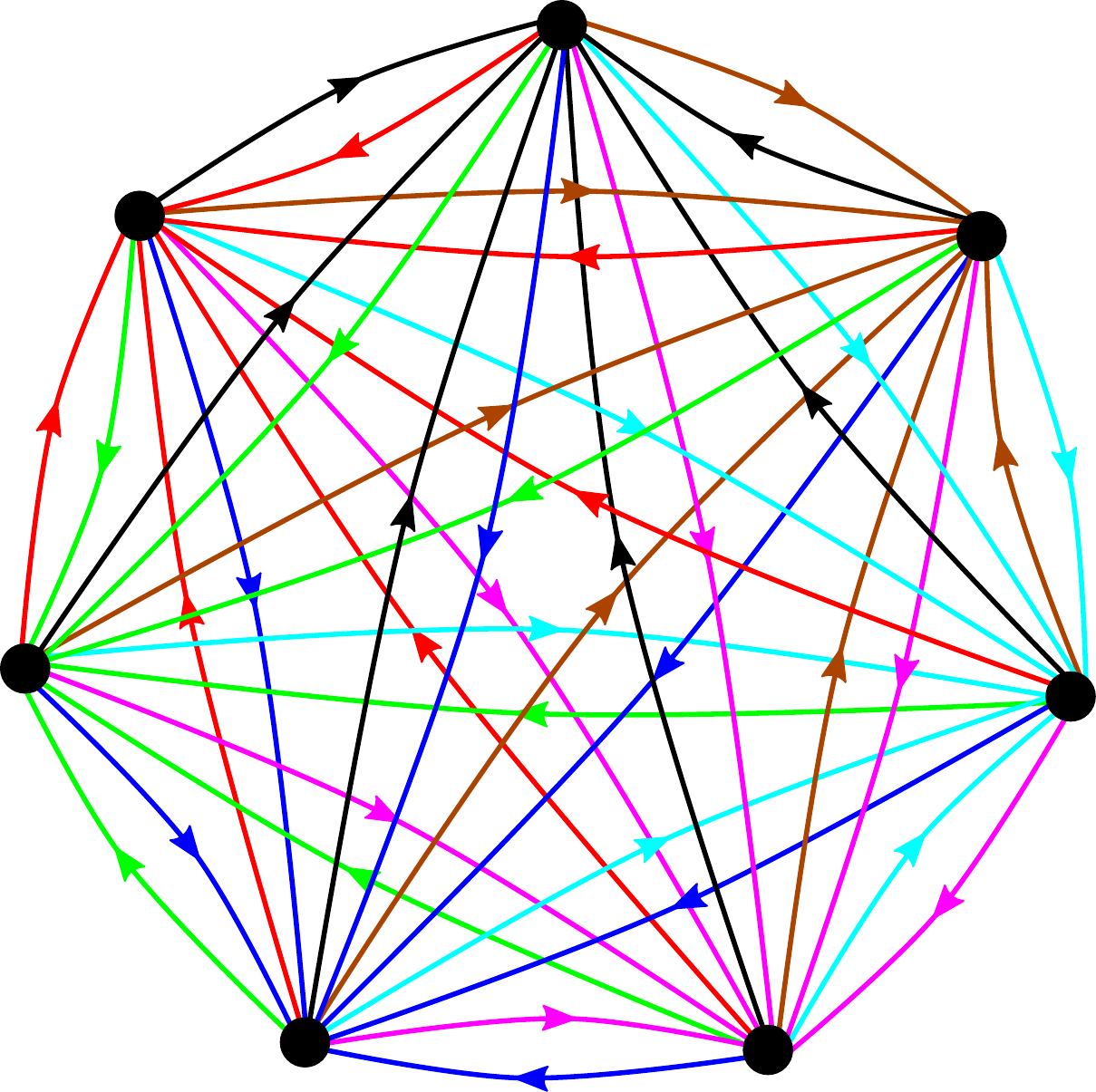}
  \caption{A labelled directed graph whose labelling is out-proper, but not in-proper. Here all the edge-labels are not vertices. The edge-labelling is such that every vertex $v$ has a ``chosen colour'' with all edges directed towards $v$ having the chosen colour. Notice that deleting all edges having a particular label reduces the in-degree of some vertex to $0$, effectively isolating it.}\label{FigureK7}
\end{figure}
When $G$ is a multigraph, then we know that $\DGM$ is out-properly labelled, but not necessarily in-properly labelled. Definition~\ref{DefinitionWeakConnectedness} isn't the right notion of connectivity for studying such graphs. It is possible to have an out-proper labelling of the complete directed graph in which any vertex can be isolated by deleting just one label.  See Figure~\ref{FigureK7} for an example of such a graph. 
This graph is a complete directed graph where every edge $xy$ is labelled by $\ell_y$ (for some label $\ell_y$ which only appears on in-going edges to $y$.)
This graph has a high out-degree but doesn't have any $(1, \infty)$-connected subgraphs.
This is the issue with using $(k,d)$-connectedness since we would like high out-degree graphs to have highly connected subsets. 

We introduce a different kind of connectedness, for which the graph in Figure~\ref{FigureK7} is highly connected.
The following is at the heart of the notion of connectivity which we use.
\begin{definition}[Reaching]
For a vertex $v\in V(D)$ and a set $R\subseteq V(D)$, we say that $v$ $(k,d,\Delta)$-reaches $R$ if for any set $S$ of $\leq k$ labels, there are length $\leq d$ switching paths avoiding $S$ to all, except possibly at most $\Delta$, vertices $x\in R$.
\end{definition}

Standard notions of connectedness are based on studying when two vertices are connected by a path. ``Reaching'' is fundamentally different from these since it is of no use to know that a vertex $u$ reaches  another vertex $v$. In fact any vertex $u$ $(\infty, \infty, 1)$-reaches any singleton $\{v\}$ (since $\Delta =1$, we can let $\{v\}$ be the set of $\Delta$ vertices in $S$ to which we don't need to find a path in the definition of reaching. More generally,  there is nothing to check in the definition of ``reaching'' when  $\Delta\geq |R|$.) Thus ``reaching'' is only meaningful when we talk about a vertex reaching a reasonably large set of vertices $R$. Notice that the graph in Figure~\ref{FigureK7} has good connectivity properties with our new definition. 
\begin{exercise}
For the labelled directed graph $D$ in Figure~\ref{FigureK7} and any $k\leq \Delta$, show that every $v\in V(D)$ $(k,1,\Delta)$-reaches $V(D)$.
\end{exercise}

To prove Theorem~\ref{TheoremDigraphIntroduction}, we will need to have a fairly deep understanding of ``reaching''. This involves first proving several basic consequences of the definition such as showing that reaching is monotone under change of parameters, preserved by unions, and has a kind of transitivity property. These properties are proved in Section~\ref{SectionBasicProperties}.

Our main goal when studying ``reaching'' will be to show that some analogue of connected components exists for the new notion of connectedness. Recall that a strongly connected component $C$ in  a directed graph is a maximal set of vertices in a graph such that for any two vertices $x$ and $y$ in $C$ there is a path from $x$ to $y$. Analogously, in a labelled graph we would like to find a maximal set  $C$ such that any $x\in C$ reaches all of $C$ for suitable parameters. This notion of a maximal reached set seems a bit hard to work with, so we will instead deal with the following approximate version.
\begin{definition}[$(k,d,\Delta,\gamma, \hat k, \hat d, \hat \Delta)$-component]
A set $C\subseteq V(D)$ is a $(k,d,\Delta,\gamma)$-component if for any vertex $v\in C$, there is  a set $R_v$ with $|R_v\symdiff C|\leq  \epsilon^3 n$ such that the following hold.
\begin{enumerate}[(i)]
\item $v$ $(k,d,\Delta)$-reaches $R_v$. 
\item $v$ doesn't $(\hat k, \hat d, \hat \Delta)$-reach any set $R$ disjoint from $R_v$ with  $|R|\geq \gamma n$.
\end{enumerate}
\end{definition}
In other words a $(k,d,\Delta,\gamma, \hat k, \hat d, \hat \Delta)$-component is a set $C$ such that every vertex $v\in C$  reaches most of $C$ and doesn't reach any large set outside $C$. It is not at all obvious that $(k,d,\Delta,\gamma, \hat k, \hat d, \hat \Delta)$-components exist for particular parameters $k,d,\Delta,\gamma, \hat k, \hat d, \hat \Delta$. An important intermediate lemma we prove in Section~\ref{SectionComponentExistence}, is that for given $k,d,\Delta,\gamma$, there is a $(k',d',\Delta',\gamma', \hat k', \hat d', \hat \Delta')$-component for new parameters $k',d',\Delta',\gamma', \hat k', \hat d', \hat \Delta'$ close to $k,d,\Delta,\gamma$.

We make a remark about how constants will be dealt with throughout this paper. Looking at the definitions of ``$(k,d,\Delta)$-reaches'' and ``$(k,d,\Delta,\gamma, \hat k, \hat d, \hat \Delta)$-component'', they  look a bit scary because of the large number of parameters there are in each definition. In the actual proof of Theorem~\ref{TheoremDigraphIntroduction} in Section~\ref{SectionDirectedGraphs} this won't be the case because we introduce a single parameter, $\lambda$, which will control each of the parameters $k,d,\Delta,\gamma, \hat k, \hat d, \hat \Delta$. Formally, in Section~\ref{SectionDirectedGraphs} we define seven explicit functions $k_{\epsilon}(\lam)$, $d_{\epsilon}(\lam)$, $\Delta_{\epsilon}(\lam)$, $\gamma_{\epsilon}(\lam)$, $\hat k_{\epsilon}(\lam)$, $\hat  d_{\epsilon}(\lam)$, and $\hat \Delta_{\epsilon}(\lam)$ depending on $\epsilon$ (which is the constant given in the statement of Theorem~\ref{TheoremDigraphIntroduction}.) Then we say that $v$ $\lam$-reaches  a set $R$ if $v$ $(k_{\epsilon}(\lam),d_{\epsilon}(\lam),\Delta_{\epsilon}(\lam))$-reaches $R$, and that a set $C$ is a $\lam$-component if $C$ is a $(k_{\epsilon}(\lam),d_{\epsilon}(\lam),\Delta_{\epsilon}(\lam), \gamma_{\epsilon}(\lam), \hat k_{\epsilon}(\lam), \hat  d_{\epsilon}(\lam) , \hat \Delta_{\epsilon}(\lam))$-component.
The advantage of this is that it means that only one parameter, $\lambda$, needs to be kept track of between the various lemmas that we prove. This  makes the high level structure of the proof of Theorem~\ref{TheoremDigraphIntroduction} easier to follow.

We mention a final definition which we use in the paper. 
\begin{definition}[Bypassing]
For a vertex $v\in V(D)$ and a set $B\subseteq V(D)$, we say that $v$ $(\hat k, \hat d,\hat \Delta, \gamma)$-bypasses $B$ if $v$ doesn't $(\hat k,\hat  d,\hat  \Delta)$-reach any set $R$ contained in $B$ with  $|R|\geq \gamma n$
\end{definition}
The significance of the above definition is that part (ii) of the definition of $(k,d,\Delta,\gamma, \hat k, \hat d, \hat \Delta)$-component can be now rephrased as ``$v$  $(\hat k, \hat d, \hat \Delta, \gamma)$-bypasses $V(D)\setminus R_v$''. 
Thus the notion of bypassing is important because it eases the study  of $(k,d,\Delta,\gamma, \hat k, \hat d, \hat \Delta)$-components.

Recall that at the start of the section we said that the reason for using connectedness is to be able to study ``amidstness''. It is not immediately apparent how the definitions we introduce do this. With a bit of work it is possible to prove that in a $(k,d,\Delta,\gamma, \hat k, \hat d, \hat \Delta)$-component $C$, most triples $(u,c,v)\in C\times C\times C$ have $c$ amidst $u$ and $v$.
\begin{lemma}\label{LemmaMengerIntroduction}
For $\epsilon>0$,  $D$ a  sufficiently large labelled directed graph,  and  $C$ a $(k,d,\Delta,\gamma,$ $\hat k,$ $\hat d,$ $\hat \Delta)$-component in $D$ for suitable $k,d,\Delta,\gamma, \hat k, \hat d, \hat \Delta$, there are at least $|C|^3-(\epsilon n)^3$ triples $(u,c,v)\in C\times C\times C$ with $c$ amidst $u$ and $v$.
\end{lemma}
The above lemma is an easy consequence of Lemma~\ref{MengerLemma} which we prove in Section~\ref{SectionComponentGrowth}. The full Lemma~\ref{MengerLemma} will say a bit more, giving information about the structure of triples $(u,c,v)\in C\times C\times C$ with $c$ amidst $u$ and $v$.

\subsection{An overview of the proof of Theorem~\ref{TheoremDigraphIntroduction}}\label{SectionProofOverview}
Here we give a high level overview of the strategy of the proof of Theorem~\ref{TheoremDigraphIntroduction}.
The proof begins by supposing for the sake of contradiction that there is a vertex $u\in V(D)$ such that for every vertex $v$ and a set of labels $A$ amidst $u$ and $v$ we have $|N^+_A(v)|<|A|-|X_0|+\epsilon  n$.
The proof of the theorem naturally splits into three parts.
\begin{enumerate}
\item Find  $(k_i,d_i,\Delta_i,\gamma_i, \hat k_i,\hat d_i,\hat \Delta_i)$-components $C_0, \dots, C_m$ for suitable parameters such that $C_i\cap C_{i+1}\neq \emptyset$. In addition we find a short switching path from $u$ to each $C_i$. This is done as follows:
\begin{enumerate}
\item[1.1] Prove lemmas along the lines of ``for any vertex $v$ and parameters $k,d,\Delta,\gamma$ there are complementary sets $R_v$ and $B_v$ such that $v$ $(k',d',\Delta')$-reaches $R_v$ and doesn't reach anything in $B_v$  for suitable parameters $k',d', \Delta'$ close to $k, d, \Delta$. This is performed in Lemmas~\ref{SetReachedBypassed1} and~\ref{SetReachedBypassed2}.
\item[1.2] Show that reaching has a transitivity property: If $v$ reaches a sufficiently large set $R$ and every vertex in $R$  reaches a set $R'$, then $v$ 
reaches $R'$. This is performed in Lemma~\ref{Transitivity}.
\item[1.3] Choose a vertex $v$ with the set $R_v$ from part 1.1 as small as possible. Using transitivity, it is possible to show that for most vertices $u\in R_v$ we have that $|R_u\symdiff R_v|$ is small. By letting $C=R_v$ minus a few vertices it is possible to get a single component of the sort we want. This is performed in Lemma~\ref{ComponentExistence}.
\item[1.4] By iterating 1.3, we can get the sequence of components  $C_0, \dots, C_m$ which we need. This is performed in Lemma~\ref{ComponentTower}.
\end{enumerate}

\item Show that  if $C$ is a $(k,d,\Delta,\gamma, \hat k,\hat d,\hat \Delta)$-component close to $u$, then \textbf{\emph{either}} any $v\in C$ $(k',d',\Delta',\gamma')$-reaches some set $R$ with $|R|\geq |C|+(\epsilon -o(1))n$ for suitable parameters \textbf{\emph{or}} the conclusion of Theorem~\ref{TheoremDigraphIntroduction} holds for some $A\subseteq C\cup X_0$. The formal statement of this is Lemma~\ref{ComponentGrowth}. This step is performed as follows:
\begin{enumerate}
\item [2.1] We show that for most triples $(u,c,v)\subseteq C\times C \times C$, $c$ is amidst $u$ and $v$. This is performed in Lemma~\ref{MengerLemma}.
\item [2.2] Let $R$ be the set of $z\in V(D)$ for which there are a lot of triples $(u,c,v)$ such that  $vz$ is an edge labelled by $c$ and $c$ is amidst $u$ and $v$. 
\item [2.3] The vertex $v$ ends up $(k',d',\Delta')$-reaching $R$ as a consequence of 2.1. This is performed in Claim~\ref{ClaimRReached}.
\item [2.4] Use the assumption of Theorem~\ref{TheoremDigraphIntroduction} applied to a suitable subset of $C\cup X_0$ together with Lemma~\ref{LemmaMengerIntroduction} to show that $|R|\geq |C|+(\epsilon -o(1))n.$ This is performed in Claims~\ref{EQBoundLZP} and~\ref{ClaimRLarge}.
\end{enumerate}

\item Combining parts 1 and 2 and the definition of $(k,d,\Delta,\gamma, \hat k,\hat d,\hat \Delta)$-component we obtain that $|C_{i+1}|\geq |C_i|+(\epsilon-o(1)) n$ for every $C_i$ from part 1. If the number of components $m\gg \epsilon^{-1}$ this gives a contradiction to $|C_m|\leq |V(D)|=n$.
\end{enumerate}

\subsection{An example}
In this section we give an illustrative labelled directed graph and explain how the proof of Theorem~\ref{TheoremDigraphIntroduction} works for that particular graph.

\begin{figure}[ht]
  \centering
    \includegraphics[width=0.87\textwidth]{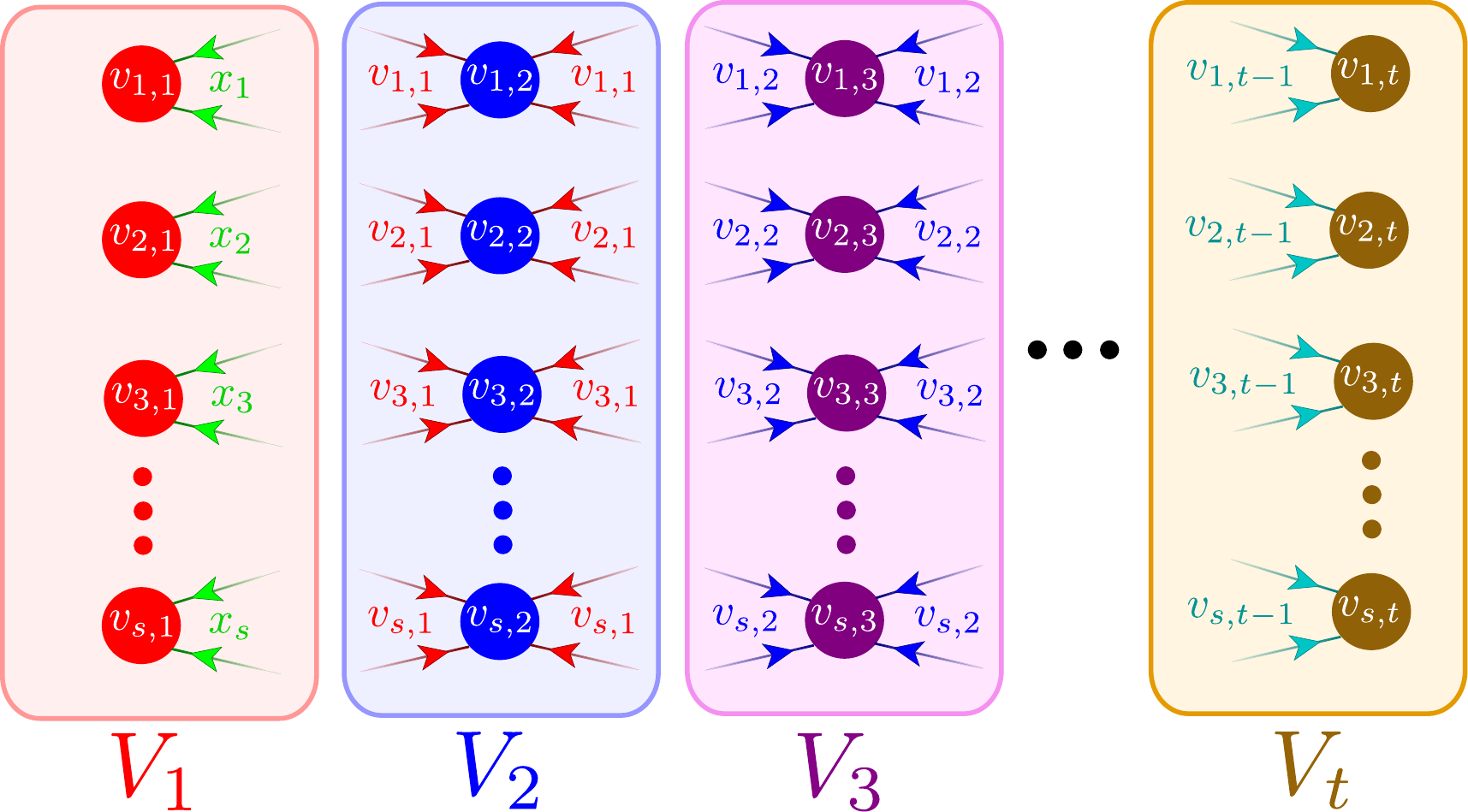}
  \caption{The labelled, directed graph $D_{s,t}$.}\label{FigureExampleGraph}
\end{figure}

For a fixed $\epsilon>0$ and $t\ll \epsilon^{-1}$, we define a directed graph $D_{s,t}$ as follows: $D_{s,t}$ has $n=st$ vertices split into $t$ disjoint classes $V_{1},\dots, V_t$ each of size $s$ with $V_i=\{v_{1,i}, \dots, v_{s,i}\}$. The set of non-vertex labels in $D_{s,t}$ is $X_0=\{x_1, \dots, x_s\}$. All the edges are present in $G$ going in both directions. Each vertex in $G$ has it's own ``chosen label'' with all edges directed towards the vertex having that label (much like the graph in Figure~\ref{FigureK7}.) For $i\neq 1$ and any vertex $u$, the edge $uv_{j,i}$ has label $v_{j, i-1}$. For a vertex $u$ and a vertex $v_{j,1}\in V_1$, the edge $uv_{j,1}$ has label $x_j$.

Because of the simple structure the graphs $D_{s,t}$ have, it is possible to describe all switching paths, reached sets, bypassed sets, and components in these graphs. 

\begin{exercise}[Switching paths in $D_{s,t}$]
A directed path $P=(p_0, p_1, \dots, p_d)$  is a switching path if, and only if,  ``$v_{j,i}= p_k\in  V(P) \implies v_{j,i-1}, \dots, v_{j, 1}\in \{p_1, \dots, p_{k-1}\}$''.
\end{exercise}

\begin{exercise}[Amidstness in $D_{s,t}$]\label{ExerciseAmidstDst}
\
\begin{itemize}
\item A vertex $v_{a,b}$ is amidst   vertices $v_{c,d}$ and $v_{e,f}$ if, and only if, $b,d\neq f$  and $v_{a,b}\neq v_{c,d}$.
\item A label $x_{b}$ is amidst vertices $v_{c,d}$ and $v_{e,f}$ if, and only if, $b,d\neq f$.
\end{itemize}
\end{exercise}

In each of the next three exercises we give a necessary condition and a sufficient condition for a set to be reached, bypassed, or be a component. Although the necessary conditions and a sufficient conditions which we give are not exactly the same, they are always quite similar. Thus the next three exercises should be seen as giving a near-characterization of sets which are reached, bypassed, and components in $D_{s,t}$.
\begin{exercise}[Reaching in $D_{s,t}$]
Let $v\in V(D_{s,t})$ and $R\subseteq V(D_{s,t})$.
\begin{itemize}
\item If $v$ $(k, d, \Delta)$-reaches $R$ then $|R\setminus (V_1\cup \dots\cup V_d)|\leq \Delta$.
\item If $|R\setminus (V_1\cup \dots\cup V_d)|\leq \Delta- (k+1)d$, then  $v$ $(k, d, \Delta)$-reaches $R$.
\end{itemize}
\end{exercise}

\begin{exercise}[Bypassing in $D_{s,t}$]
Let $v\in V(D_{s,t})$, $B\subseteq V(D_{s,t})$, and $\hat\Delta\geq (\hat k+1)\hat d$.
\begin{itemize}
\item If  $v$ $(\hat k, \hat d, \hat \Delta, \gamma)$-bypasses $B$ then $|B\cap (V_1\cup \dots\cup V_d)|\leq  \gamma n$.
\item If $|B\cap (V_1\cup \dots\cup V_d)|< \gamma n-\hat\Delta$, then  $v$ $(\hat k, \hat d, \hat \Delta, \gamma)$-bypasses $B$.
\end{itemize}
\end{exercise}

\begin{exercise}[Components in $D_{s,t}$]\label{ExerciseComponentsDst}
Let $C\subseteq V(G)$, $\epsilon^3 n\geq \max(\gamma n, \Delta,   \hat \Delta$), $\Delta\geq (k+1)d$, $\hat \Delta\geq (\hat k+1)\hat d$, and $\hat d\leq d$
\begin{itemize}
\item If $C$ is a $(k,d,\Delta,\gamma,$ $\hat k,$ $\hat d,$ $\hat \Delta)$-component then $|C\symdiff  (V_1\cup \dots\cup V_d)|\leq 8\epsilon^3 n$.
\item If $|C\symdiff  (V_1\cup \dots\cup V_d)|\leq \epsilon^3 n$ then $C$ is a $(k,d,\Delta,\gamma,$ $\hat k,$ $\hat d,$ $\hat \Delta)$-component.
\end{itemize}
\end{exercise}

Using Exercise~\ref{ExerciseAmidstDst} we can check that Thereom~\ref{TheoremDigraphIntroduction} holds for the graphs $D_{s,t}$. To see this notice that for any pair of vertices $v_{a,b}$ and $v_{c,d}$ for $b\neq d$ and any set of labels $A$ with $A\cap \{v_{i,b}, v_{i,d}:i=1, \dots, s\}=\emptyset$ we have $A$ amidst $v_{a,b}$ and $v_{c,d}$. Notice also that $|N^+_A(v)|=|A|$ or $|A|-1$ for every vertex $v$ and set of labels $A$. Finally, recall that $|X_0|=s\geq\epsilon n$.
Thus we see that for a given vertex $u=v_{a,b}$ the conclusion of Thereom~\ref{TheoremDigraphIntroduction} holds by choosing $v=v_{c,d}$ for $d\neq b$ and $A$ any set of labels disjoint from $\{v_{i,b}, v_{i,d}:i=1, \dots, s\}$ with $|A|>n$ (for example we could take $A=V(D_{s,t})\cup X_0\setminus \{v_{i,b}, v_{i,d}:i=1, \dots, s\}$.) 

This example teaches us an important lesson. It is not hard  to see that if $A$ is amidst $u$ and $v$ and satisfies $|N^+_A(v)|<|A|-|X_0|+\epsilon  n$ then $|A|>n$ must hold. From this we see that any proof of Theorem~\ref{TheoremDigraphIntroduction} for the graph $D_{s,t}$ must ``find'' a very large set of labels $A$ amidst some pair of vertices. In the remainder of this section, we explain how the strategy in Section~\ref{SectionProofOverview} finds such a set $A$.

In Part 1 of the overview in Section~\ref{SectionProofOverview}, the proof finds  $(k_i,d_i,\Delta_i,\gamma_i, \hat k_i, \hat d_i,\hat\Delta_i)$-components $C_0, \dots, C_m$ for suitable parameters such that $C_i\cap C_{i+1}\neq \emptyset$. Here ``suitable parameters'' means that $m\gg \epsilon^{-1}$, that $d_i, \Delta_i, \gamma_i, \hat d_i$, and $\hat\Delta_i$ increase with $i$ while $k_i$ and $\hat k_i$ decrease with $i$. For the sake of argument, let us consider what happens when $m=\epsilon^{-4}$, $d_i=\hat d_i=i$, $\gamma_i= \epsilon^9 i$, $\Delta_i= \hat \Delta_i=2im$, and $k=\hat k=m-i$. 
For these values, Exercise~\ref{ExerciseComponentsDst} tells us that for $i=1, \dots, t$ we must have $|C_i\symdiff (V_{1}\cup \dots\cup V_i)|\leq 8\epsilon^3 n$ and that for $i=t+1, \dots, m$ we have $|C_i\symdiff (V_{1}\cup \dots\cup V_t)|\leq 8\epsilon^3 n$. 

In Parts 2 and 3 of the overview in Section~\ref{SectionProofOverview}, it is shown that if $C$ is a component close to $u$, then {\emph{either}} any $v\in C$ $(k',d',\Delta',\gamma')$-reaches some set $R$ with $|R|\geq |C|+(\epsilon -o(1))n$ for suitable parameters {\emph{or}} the conclusion of Theorem~\ref{TheoremDigraphIntroduction} holds for some $A\subseteq C\cup X_0$ and $v\in C$. Testing this for the components $C_0, \dots, C_m$ from Part 1, we see that for $i=0, \dots, t-1$, there is a set $R$ such that  $v$ $(k_i,d_i+1,\Delta_i,\gamma_i)$-reaches some set $R$ with $|R|\geq |C|+(\epsilon -o(1))n$ (namely we can take $R=V_1\cup\dots \cup V_{i+1}$.) On the other hand for $i=t, \dots, m$ such a set $R$ doesn't exist, so Part 2 would imply that the conclusion of Theorem~\ref{TheoremDigraphIntroduction} holds for some $A\subseteq C_i\cup X_0$ and $v\in C$. If $u=v_{a,b}$ then we see that this is indeed the case with e.g., $A=C_i\cup X_0\setminus\{v_{1,b}, \dots, v_{t,b}, v_{1,c}, \dots, v_{t,d}\}$ and $v=v_{a,c}$.

\subsection*{Notation}
For standard notation we follow~\cite{BollobasModernGraphTheory}.
A path $P=(p_0, p_1, \dots, p_d)$ in a directed graph $D$ is a sequence of vertices $p_0, p_1, \dots, p_d$ such that $p_ip_{i+1}$ is an edge for $i=0, \dots, d-1$. The \emph{order} of $P$ is the number of vertices it has, and the \emph{length} of $P$ is the number of edges it has.
We'll use additive notation for concatenating paths i.e. if $P=(p_1,p_2,\dots, p_i)$ and $Q=(p_i,p_{i+1}, \dots, p_d)$ are two internally vertex-disjoint paths, then we let $P+Q$ denote the path $(p_1p_2\dots p_d)$.
Throughout the paper, all directed graphs will be simple meaning that an edge $xy$ appears only at most once. We do allow both of the edges $xy$ and $yx$ to appear in the directed graphs we consider.
For clarity we will omit floor and ceiling signs where they aren't important.

Our digraphs are always simple i.e. they never have two copies of an edge going from a vertex $u$ to a vertex $v$.
A digraph is out-properly labelled if all out-going edges at a vertex have different labels. For a vertex $v$ in a digraph, the out-neighbourhood of $v$, denoted $N^+(v)$ is the set of $w\in V(D)$ with $vw$ an edge of $D$.

Throughout the paper we will deal with edge-\emph{coloured} undirected graphs and edge-\emph{labelled} directed graphs. 
The difference between the two concepts is that in an edge-coloured graph, the set of possible colours is just some ambient set, whereas in an edge-labelled digraph $D$ the set of possible labels is $V(D)\cup X_0$ where $V(D)$ is the set of vertices of $D$ (and $X_0$ is some ambient set unrelated to $D$.)
Formally, an edge-labelled  directed graph is defined to be a directed graph $D$ together with a set $X_0$ with $X_0\cap V(D)=\emptyset$ and a labelling function $f:E(D)\to V(D)\cup X_0$.
The set $X_0$ is called the set of \emph{non-vertex labels} in $D$. We call $X_0\cup V(D)$ the \emph{set of labels} in $D$ (regardless of whether $D$ actually has edges labelled by all elements of $X_0\cup V(D)$).

Throughout the paper we will always use  ``$G$'' to denote a coloured bipartite graph with parts $X$ and $Y$ and $M$ a rainbow matching in $G$.  We'll use $C_G$  to denote the set of colours in $G$ and $C_M$ to denote the set of colours in $M$. 
We'll use $V(M)$ to mean the set of vertices contained in edges of $M$,
and $X_0=X\setminus V(M)$ and $Y_0=Y\setminus V(M)$ to denote the vertices in $X$ and $Y$ outside $M$.
For a colour $c\in C_M$, we use $m_c$ to denote the colour $c$ edge of $M$.

\section{From bipartite graphs to directed graphs}\label{SectionBipartite}
In this section we show how go from the Aharoni-Berger Conjecture to a problem about edge-labelled digraphs.
We define a directed, edge-labelled digraph $\DGM$ corresponding to a coloured bipartite graph $G$ and a rainbow matching $M$ in $G$. 
\begin{definition}[The directed graph $\DGM$]\label{DefinitionDGM}
Let $G$ be a coloured bipartite graph with parts $X$ and $Y$ and $n$ colours. Let $M$ be a rainbow matching in $G$. Let $X_0=X\setminus V(M)$ be the subsets of $X$  disjoint from $M$. Let $C_G$ be the set of colours used in $G$ and $C_M\subset C_G$ be the set of colours used on edges in $M$. For a colour $c\in C_M$, we let $m_c$ denote the colour $c$ edge of $M$.
The labelled digraph $\DGM$ corresponding to $G$ and $M$ is defined as follows:
\begin{itemize}
\item The vertex set of $\DGM$ is the set $C_G$.  
\item The edges of $\DGM$ are be labelled by elements of the set $X_0\cup C_M$.  
\item For two colours $u$ and $v\in V(\DGM)$ and a vertex $x \in X$, there is a directed edge from $u$ to $v$ in $\DGM$ whenever $v\in C_M$ and there is a colour $u$ edge from $x$ to  $m_v\cap Y$. 
\begin{itemize}
\item If $x\in X_0$ then the edge $uv$ is labelled by $x$.
\item If $x\in m_c\in M$ then $uv$ is labelled by $c$, the colour of $m_c$.
\end{itemize}
\end{itemize}
\end{definition}
Notice that every edge  $e\in E(G)$ corresponds to at most one edge of $\DGM$.
There are two types of edges in $G$ which do not correspond to edges of $\DGM$: Edges going through $Y_0$ do not appear in $\DGM$, and also the edges of $M$ do not appear in  $\DGM$ either.
Thus the edges of $\DGM$ are naturally in bijection with the edges of $G[X\cup V(M)]\setminus M$.
  Also notice that if $c\not\in C_M$ is a colour which doesn't appear in $M$, then the in-degree of $c$ in $\DGM$ is zero.

For any set $L$ of labels in $\DGM$ we define a corresponding set $(L)_X$ of vertices in $X$ as follows.
For a colour $c\in C_M$ we define $(c)_X$ to be $m_c\cap X$ where $m_c$ is the colour $c$ edge of $M$. For any vertex $x\in X_0$, we set $(x)_X=\{x\}$. For $L$ a set of labels of $\DGM$, we define $(L)_X=\bigcup_{\ell\in L} (\ell)_X$ i.e. $(L)_X$ is the subset of $L$ consisting of vertices in $X_0$  together with  $M'\cap X$ where $M'$ is the subset of $M$ consisting of edges whose colour is in $L$.

Notice that with the above definition, if $xy$ is an edge of $G$ and $\ell$ is the label of the corresponding edge of $\DGM$, then we always have $(\ell)_X=x$. Conversely if $uv$ is an edge of $\DGM$ labelled by $\ell$, then the corresponding edge of $G$ goes from $(\ell)_X$ to $m_v\cap Y$. Also notice that $|(L)_X|=|L|$ for any set of labels of $\DGM$.

It turns out that if $G$ is properly coloured, then $\DGM$ is out-properly labelled and simple.
\begin{lemma}\label{LemmaProperlyColoured}
Let $G$ be a properly edge-coloured bipartite graph and $M$ a matching in $G$.
Then the directed graph $\DGM$ is out-properly labelled and simple. 
%An edge $uv\in \DGM$ is never labelled by $u$.
\end{lemma}
\begin{proof}
Suppose that $uv$ and $uv'$ are two distinct edges of $\DGM$ with the same label $\ell$.
By definition of $\DGM$ they correspond to two  edges of the form $(\ell)_Xy$ and $(\ell)_Xy'$ of $G$ having colour $u$, where $y=m_v\cap Y$ and $y'=m_{v'}\cap Y$. But this contradicts the colouring of $G$ being proper.

Suppose that $\DGM$ is not simple i.e. an edge $uv$ occurs twice with different labels $\ell$ and $\ell'$. This corresponds to two edges of the form $(\ell)_Xy$ and $(\ell')_Xy$ of $G$ having the same colour $u$ (where $y =m_v\cap Y$.) But this contradicts the colouring of $G$ being proper.
%Suppose that $uv$ is labelled by $u$. Since $uv$ is labelled by $u$, we have $a$ colour $u$ edge in $G$ between $m_u\cap X$ and $m_v\cap Y$. Then this edge and $m_u$ intersect in $G$  and have the same colour, contradicting colouring of $G$ being proper.
\end{proof}

We now come to the central objects of study in this paper---switching paths. Switching paths in a labelled digraph $D$ are rainbow paths which have a kind of ``consistency'' property for the edges they contain which are labelled by vertices of $D$.
\begin{definition}[Switching path]
A path $P=(p_0, \dots, p_d)$ in an edge-labelled, directed graph $D$ is a switching path if the following hold.
\begin{itemize}
\item $P$ is rainbow i.e. the edges of $P$ have different labels.
\item If $p_{i}p_{i+1}$ is labelled by a vertex $v\in V(D)$, then $v=p_j$ for some $1\leq j\leq i$.
\end{itemize}
\end{definition}

Another key definition in this paper is of a label being \emph{amidst} two vertices. 
\begin{definition}[Amidst]\label{DefinitionAmidst}
Let $u$ and $v$ be two vertices in an edge-labelled, directed graph $D$, and $c$ a label. We say that $c$ is amidst $u$ and $v$ if there is a switching path $P=(u,p_1,\dots, p_{d},v)$ from $u$ to $v$ such that the following hold.
\begin{itemize}
\item There are no edges of $P$ labelled by $c$.
\item If $c$ is a vertex of $D$ then   $c\in \{p_1, \dots, p_d, v\}$.
\end{itemize}
\end{definition}
If $P$ is a path as in Definition~\ref{DefinitionAmidst}, then we say that $P$ \emph{witnesses} $c$ being amidst $u$ and $v$.
Notice that like in the definition of ``switching path'', in the second part of the definition of ``amidst'' the vertex $c$ is required to be a \emph{non-starting} vertex of $P$. Also notice that if there is a switching path $P$  from $u$ to $v$ with $|P|\geq 2$, then $v$ is amidst $u$ and $v$, as witnessed by $P$.

The following lemma establishes a link between a matching $M$ being maximum in a graph $G$ and the behavior of switching paths in the corresponding digraph $\DGM$. 
\begin{lemma}\label{LemmaSwitchingPathVertexAmidst}
Let $G$ be a properly coloured bipartite graph with parts $X$ and $Y$ and $M$ a maximum rainbow matching in $G$.

Suppose that $M$ misses a colour $c^*$ and $a$ is a  label in $\DGM$ which is amidst $c^*$ and some $v\in V(\DGM)$. Then there is no  colour $v$ edge in $G$ from $(a)_X$ to $Y_0=Y\setminus V(M)$.
\end{lemma}
\begin{proof}
Suppose for the sake of contradiction that a colour $v$ edge $(a)_Xy$ exists for $y\in Y_0$.
Let $P$ be a switching path witnessing $a$ being amidst $c^*$ and $v$. Let $p_0, p_1, \dots, p_k$ be the vertex sequence of $P$ with $p_0=c^*$ and $p_k=v$. For $1\leq i\leq k$, let $m_i$ be the edge of $M$ with colour $p_i$. Such edges exist since the in-degree of $p_i$ is positive for $i\geq 1$.
For $0\leq i\leq k-1$ let $\ell_i$ be the label of $p_{i}p_{i+1}$ and define $x_i=(\ell_i)_X$.
For $0\leq i\leq k-1$ let $e_i$ be the edge of $G$ corresponding to the edge $p_{i}p_{i+1}$ of $\DGM$ i.e. $e_i$ is the colour $p_{i}$ edge going from $x_i$ to $m_{i+1}\cap Y$.
\begin{claim}\label{ClaimSwitchingPathExchangeEdges}
Let $M'=M\cup \{e_0, \dots, e_{k-1}\}\setminus\{m_1, \dots, m_k\}$. Then $M'$ is a rainbow matching in $G$ of size $|M|$ missing the colour $v$.
\end{claim}
\begin{proof}
First we show that $M'$ is a rainbow set of edges missing the colour $v$. Notice that for each $i\geq 1$, $e_i$ and $m_i$ both have colour $p_i$. %Also $ty$ has the same colour as $m_k$, and the edge $e_0$ has colour $c^*$. 
Also the edge $e_0$ has colour $p_0=c^*$.
Since $M$ is rainbow and missed colour $c^*$, $M\setminus\{m_1, \dots, m_k\}$ is rainbow and misses the colours $c^*$, $p_1, \dots, p_k$. Therefore $M'$ is rainbow and misses  colour $p_k=v$.

It remains to show that $M'$ is a matching. Notice that $M\setminus\{m_1, \dots, m_k\}$ is a matching as a consequence of $M$ being a matching.

Next we show that $\{e_0, \dots, e_{k-1}\}$ is a matching.
Since $P$ is a switching path, its edges have different labels, which is equivalent to the vertices $x_0, \dots, x_{k-1}$ being distinct. Also, since $P$ is a path, the vertices $p_1, \dots, p_k$ are distinct which implies that the edges $m_1, \dots, m_k$ are also distinct. Since $e_i$ goes from $x_i$ to $m_{i+1}\cap Y$, these imply that for distinct $i$ and $j$ we have $e_i\cap e_j=\emptyset$.

Finally we show that $e_i\cap m=\emptyset$ for $0\leq i\leq k-1$ and $m\in  M\setminus\{m_1, \dots, m_k\}$. Suppose that $e_i\cap m\cap Y\neq \emptyset$. Then since $e_i\cap Y=m_{i+1}\cap Y$, we have $m=m_{i+1}$ which contradicts $m\in  M\setminus\{m_1, \dots, m_k\}$.  
Suppose that $e_i\cap m\cap X\neq \emptyset$, or equivalently $e_i\cap m\cap X=\{x_i\}$. Then $x_i\in V(M)$ which is equivalent to $\ell_i\in V(\DGM)$. In particular we find out that $\ell_i$ is a colour in $G$ (rather than a vertex of $X_0$.) Recall that $\ell_i$ is the label of the edge $p_ip_{i+1}$ of $P$. Using the definition of $P$ being a switching path, we get that $\ell_i=p_j$ for some $1\leq j\leq i$. Then $m_j$ is the colour $\ell_i$ edge of $M$ which gives $x_i=(\ell_i)_X\in m_j\cap X$. This implies that $m=m_j$ which contradicts $m\in M\setminus\{m_1, \dots, m_k\}$.
\end{proof}

We claim that $(a)_X\not\in V(M')$.
Since $a$ is amidst $c^*$ and $v$, we have that $a$ doesn't appear on edges of $P$, which implies $(a)_X\neq x_i$ for $0\leq i\leq k-1$. This shows that  $(a)_X$ is disjoint from $e_0, \dots, e_{k-1}$.
If $a\not\in C_M$,  then we  have $a\in X_0$ and so  $(a)_X\not\in V(M)$ which gives  $(a)_X\not\in V(M')$.
If $a\in C_M$, then since $a$ is amidst $c^*$ and $v$, we have that $a=p_i$ for some $1\leq i\leq k$. 
This gives $(a)_X=m_i\cap X$ which implies that $(a)_X\not\in V(M\setminus\{m_1, \dots, m_k\})$  and hence $(a)_X\not\in V(M')$.

Now we have that neither of the vertices $(a)_X$ or $y$ are in $M'$, and also $M'$ misses colour $v$. Thus $M'+(a)_Xy$ is a rainbow matching of size $|M'|+1$ contradicting the maximality of $M$.
\end{proof}

For a set of labels $L$ in a labelled digraph $D$, define $$N_{L}^+(v)=\{w\in N^+(v): vw \text{ is labelled by some } \ell\in L\}.$$ Notice that from the definition of $\DGM$, we have that $|N_{L}^+(v)|$ is exactly the number of colour $v$ edges in $G$ going from $(L\setminus \{v\})_X$ to $Y\cap V(M)$.

The following corollary of the above lemma shows that if we have a graph $G$ with a maximum matching missing some colour, then the corresponding digraph $\DGM$ satisfies a  degree condition. 
\begin{lemma}\label{LemmaMaximumMatchingExpansion}
Let $G$ be a properly coloured bipartite graph with parts $X$ and $Y$, $M$ a maximum rainbow matching in $G$, and $X_0=X\setminus V(M)$.

Suppose that $M$ misses a colour $c^*$, $v$ is a colour in $G$ with $|M|+k$ edges, and $A$ is a set of labels in $\DGM$ which are amidst $c^*$ and $v$. Then $|N^+_A(v)|\geq |A|-|X_0|+k-1$.
\end{lemma}
\begin{proof}
Suppose for the sake of contradiction that $|N^+_A(v)|<|A|-|X_0|+k-1$.
Since $|(A)_X|=|A|$, there are exactly $|M|+|X_0|-|A|$ vertices in $X$ outside $(A)_X$. The number of  colour $v$ edges touching $(A\setminus\{v\})_X$ in $G$ is  $\geq |M|+k-|X\setminus (A\setminus\{v\})_X|\geq |M|+k-|X\setminus (A)_X|-1 = |A|-|X_0|+k-1>|N^+_A(v)|$. Since $|N^+_A(v)|$ equals the number of colour $v$ edges between $(A\setminus \{v\})_X$ and $Y\cap V(M)$ we obtain that there is a colour $v$ edge from some $(a)_X\in (A\setminus \{v\})_X$ to $y\in Y_0=Y\setminus V(M)$.
But, by definition of $A$ we have $a$ amidst $c^*$ and $v$, contradicting Lemma~\ref{LemmaSwitchingPathVertexAmidst}.
\end{proof}
The above lemma produces a directed graph with a degree condition. In the remainder of the paper we show that this degree condition is almost too strong to hold.  We show that if $k$ is linear in $|G|$ for every vertex $v$  and $|G|$ is sufficiently large, then no digraphs satisfying the conclusion of Lemma~\ref{LemmaMaximumMatchingExpansion} exists. This is equivalent to there being no graphs with a maximum matching $M$ satisfying the assumptions of Lemma~\ref{LemmaMaximumMatchingExpansion} i.e. we obtain that any maximum matching in such a graph must use every colour.

\section{Connectivity of labelled, directed graphs}\label{SectionDirectedGraphs}
The goal of this section is to prove the following theorem. Together with Lemmas~\ref{LemmaProperlyColoured} and~\ref{LemmaMaximumMatchingExpansion}, it immediately implies Theorem~\ref{MainTheorem}.
\begin{theorem}\label{TheoremDigraph}
For all $\epsilon$ with $0<\epsilon\leq 0.9$, there is a $N_0=N_0(\epsilon)$ such that the following holds. Let $D$ be any out-properly edge-labelled, simple, directed graph on $n\geq N_0$ vertices. Let $X_0$ be the set of labels which are not vertices of $D$

Then for all $u\in V(D)$, there is a vertex $v$ and a set of labels $A$ amidst $u$ and $v$, such that $|N^+_A(v)|<|A|-|X_0|+\epsilon  n$.
\end{theorem}

Throughout this section, for a set of vertices $S$ in a graph $D$ we denote the vertex-complement of $S$ by $\overline{S}=V(D)\setminus S$.

For a path $P$,  define a corresponding set of labels $\underline{P}$ consisting of labels which are either vertices of $P$ or labels of edges of $P$. Formally $\underline P=V(P)\cup \{\ell: \text{$\ell$ is the label of some $e\in E(P)$}\}$ denotes the set of labels consisting  of $V(P)$ together with the set of labels of edges of $P$.
For a path of length $d$, we will often use the bound $|\underline{P}|\leq 2d+1\leq 3d$.
For a set of labels $S$ and a path $P$ starting at a vertex $v$, we say that $P$ \emph{avoids} $S$ if $S\cap \underline{P}\subseteq \{v\}$ i.e. $P$ has no edges labelled by elements of $S$ and $P$ has no vertices in $S$ \emph{except possibly the starting vertex $v$}.
% For two paths $P_1$ and $P_2$, we say that $P_1$ \emph{avoids} $P_2$ if $P_1$ avoids $\underline{P_2}$. 
 
 The condition that $S$ is allowed to contain the starting vertex of $P$ in the definition of ``avoids'' may seem strange. 
We have this condition since it makes many of the arguments in this paper neater. 
 In particular, it allows us to cleanly concatenate switching paths with the following lemma.
 \begin{lemma}\label{LemmaConcatenatePaths}
Let $P=(p_0,p_1, \dots, p_t)$  and $Q=(p_t,p_{t+1}, \dots, p_s)$ be two switching paths in a labelled digraph $D$. If $Q$ avoids $\underline P$ then $P+Q=(p_1, p_2, \dots, p_s)$ is also a switching path. 
 \end{lemma}
\begin{proof}
To see that $P+Q$ is rainbow, notice that $P$ and $Q$ are rainbow and that $Q$ shares no edge-labels with $P$ since $Q$ avoids $\underline P$.
To see the second part of the definition of $P+Q$ being a switching path notice that if $p_ip_{i+1}$ is labelled by $v\in V(D)$,  then depending on whether $v\in P$ or $v\in Q$ we have $v=p_j$ for $1\leq j\leq i$  or  $v=p_j$ for $t+1\leq j\leq i$.
\end{proof}
Another consequence of the definition of ``avoids'' is that for any set of labels $S$, a single vertex path $P=v$ is a path from $v$ to $v$ avoiding $S$.

The proof of Theorem~\ref{TheoremDigraph} involves lots of constants. 
The first constant which we use is $\epsilon $ which is the constant given to us by Theorem~\ref{TheoremDigraph}. Throughout the section it is best to fix $\epsilon$ with $0<\epsilon\leq 0.9$, and to read everything that we do as a proof of Theorem~\ref{TheoremDigraph} for that particular $\epsilon $.

Next we introduce three numbers $N_0$, $\lam_\MAX$ and $\delta$ depending on $\epsilon$ whose relationship is $N_0^{-1}\ll \lam_\MAX^{-1}\ll \delta\ll \epsilon$.
The number $N_0$ will be the $N_0$ in Theorem~\ref{TheoremDigraph}, while  $\lam_\MAX$ and $\delta$ are just two numbers with no special meaning.
We set $\lam_\MAX=4^{\epsilon ^{-9}}$ and $\delta=\epsilon ^{3}$.
For an integer $x$ let $\twr(x)=\lam_\MAX^{\lam_\MAX^{\iddots^{\lam_\MAX}}}$ be the  tower function, where there are $x$ exponentiations. Set $N_0=\twr(2\lam_\MAX)$. 

Next for any $\lam\in \mathbb{N}$, we define four numbers $d_\lam, k_\lam, \Delta_\lam$, and $\gamma_\lam$. These numbers will control the variables in our definition of connectedness and will allow us to define ``reaching'' and ``components'' using just one parameter (rather than using four as in Section~\ref{SectionProofSketchConnectedness}.)
The specific definitions of $d_\lam, k_\lam, \Delta_\lam$, and $\gamma_\lam$ are not too important---the intuition is that for any $\lam  \in [1,\lam_{\MAX}]$ we have
\begin{align*}
\lam_\MAX\ll &d_{\lam} \ll k_{\lam}\ll \Delta_{\lam}\ll \gamma_{\lam}^{-1}\ll N_0,\\
d_{\lam+1}&\gg d_\lam,\hspace{1cm}
k_{\lam+1}\ll k_\lam,\\
\Delta_{\lam+1}&\gg \Delta_\lam,\hspace{0.9cm}
\gamma_{\lam+1}\ll \gamma_\lam.\\
\end{align*}
Notice that some sort of upper bound on $\lam$ is necessary for all the above to hold since $d_{\lam+1}\gg d_\lam,
$ $k_{\lam+1}\ll k_\lam,$ and  $d_{\lam} \ll k_\lam$ cannot simultaneously hold for all $\lam\in \mathbb{N}$.
Because of this, in all our lemmas we will make sure that $\lam$ is in the range $1\leq \lam\leq  \lam_{\MAX}$.

For specific $d_\lam, k_\lam, \Delta_\lam$, and $\gamma_\lam$ with which our proofs work, define
\begin{align*}
d_\lam&=\lam_\MAX\cdot 4^{\lam},\\
k_\lam&=\lam_\MAX(4^{\lam_\MAX^{4}}- 4^{\lam_\MAX \cdot\lam}),\\
\Delta_\lam&=\twr(\lam),\\
\gamma_\lam&=\twr(\lam+2)^{-1}.\\
\end{align*}

To prove Theorem~\ref{TheoremDigraph} we will need a careful understanding of the switching paths in a labelled digraph. We will study switching paths via a new notion of connectedness which we now introduce. 
The following is the heart of the notion of connectedness that we study.
\begin{definition}[$\lam$-reaching]
For a vertex $v$ in a labelled digraph $D$ and a set $R\subseteq V(D)$, we say that $v$ $\lam$-reaches $R$ if for any set $S$ of $\leq k_\lam$ labels, there are length $\leq d_\lam$ switching paths avoiding $S$ to all, except possibly at most $\Delta_\lam$, vertices $x\in R$.
\end{definition}
%Standard notions of connectedness are based on studying when two vertices are connected by a path. ``Reaching'' is fundamentally different from these since it is of no use to know that a vertex $u$ reaches  another vertex $v$. In fact any vertex $u$ $1$-reaches any singleton $\{v\}$ (since $\Delta_1\geq 1$, there is nothing to check in the definition of ``reaching'' when $R$ is just a singleton.) Thus ``reaching'' is only meaningful when we talk about a vertex reaching a set of vertices $R$.

Notice that $v$ $\lam$-reaches a set $R$ exactly when it $(k_\lam, d_\lam, \Delta_\lam)$-reaches $R$, as defined in Section~\ref{SectionProofSketchConnectedness}.
To complement the notion of ``reaching'' we introduce a notion of ``bypassing''. Informally a set $B$ is bypassed by a vertex $v$ if $v$ doesn't reach anything big inside $B$.
\begin{definition}[$\lam$-bypassing]
For a vertex $v$ in a labelled digraph $D$ and a set $B\subseteq V(D)$, we say that $v$ $\lam$-bypasses $B$ if $v$ doesn't $\lam$-reach any $R\subseteq B$ with $|R|\geq \gamma_\lam |D|$.
\end{definition}

The third key definition is that of a  $\lam$-component. Recall that when studying ordinary undirected graphs a connected component $C$ is a set where every pair $x,y\in C$ is connected by a path, and no pair $x\in C, z\not\in C$ is connected by a path. Intuitively a $\lam$-component is similar to this, with ``every pair'' replaced by ``almost every pair'' and ``no pair'' replaced by ``almost no pair''.
\begin{definition}[$\lam$-component]
A set of vertices $C$ in a labelled digraph $D$ is a $\lam$-component if for any vertex $v\in C$, there is  a set $R_v\subseteq V(D)$ with $|R_v\symdiff C|\leq \delta n$ such that the following hold.
\begin{enumerate}[(i)]
\item $v$ $\lambda$-reaches $R_v$. 
\item $v$ $(\lambda-3)$-bypasses $\overline{R_v}$.
\end{enumerate}
\end{definition}
Notice that for a labelled digraph $D$ it is far from clear that $\lam$-components exist in $D$. Section~\ref{SectionComponentExistence} will be devoted to proving that every properly labelled digraph $D$, has $\lam$-components for suitable $\lam$.

\subsection{Basic properties}\label{SectionBasicProperties}
Here we establish many basic properties of $\lam$-reaching, $\lam$-bypassing, and $\lam$-components.
The first property is that reaching or bypassing a set $W$ is preserved by passing to a subset of $W$, and by moving $\lam$ in a suitable direction.
\begin{observation}\label{ObservationHereditary}
Let $D$ be a labelled digraph, $v\in V(D)$, $R, B\subseteq V(D)$, and $\lam\in \mathbb{N}$.
\begin{enumerate}[(i)]
\item \textbf{Monotonicity of reaching:} Let $\lam^+\geq \lam$ and $R^-\subseteq R$. Then $v$ $\lam$-reaches $R$ $\implies$ $v$ $\lam^+$-reaches $R^-$.
\item \textbf{Monotonicity of bypassing:} Let $\lam^-\leq \lam$  and $B^-\subseteq B$. Then $v$ $\lam$-bypasses $B$ $\implies$ $v$ $\lam^-$-bypasses $B^-$.
\end{enumerate}
\end{observation}
\begin{proof}
For (i), let $S$ be a set of $k_{\lam^+}$ labels. Since $k_{\lam^+}\leq k_\lam$ and  $v$ $\lam$-reaches $R$, there are length $\leq d_\lam$ switching paths from $v$ to all except at most $\leq \Delta_\lam$ vertices in $R$. Since $d_{\lam^+}\geq d_\lam$ and $\Delta_{\lam^+}\geq \Delta_\lam$, these same paths give length $\leq d_{\lam^+}$ switching paths from $v$ to all except at most $\leq \Delta_{\lam^+}$ vertices in $R^-\subseteq R$.

For (ii), let $R$ be a subset of $B^-$ which is $\lam^-$-reached by $v$. Since $\lam\geq \lam^-$, by part (i) we know that $v$ $\lam$-reaches $R$. Since $R\subseteq B^-\subseteq B$ and $v$ $\lam$-bypasses $B$, we get that $|R|\leq \gamma_\lam |D|$. Since $\gamma_{\lam^-}\geq \gamma_\lam$ we get that $|R|\leq \gamma_{\lam^-}|D|$. Since $R$ was an arbitrary subset of $B^-$ which is $\lam^-$-reached by $v$, we have proved that $v$ $\lam^-$-bypasses $B^-$.
\end{proof}
Since the above observation is extremely fundamental and basic we will not always explicitly refer to it throughout its many applications. 
The next observation provides trivial conditions for sets to be reached or bypassed by a vertex.
\begin{observation}\label{ObservationReachBypassExistence}
Let $D$ be a  labelled digraph, $X_0$ the set of non-vertex labels in $D$, $v\in V(D)$, $R, B\subseteq V(D)$, and $\lam\in \mathbb{N}$.
\begin{enumerate}[(i)]
\item\textbf{Reaching small sets:} $|R|\leq \Delta_\lam$ $\implies$ $v$ $\lam$-reaches $R$.
\item\textbf{Bypassing small sets:} $|B|< \gamma_\lam|D|$ $\implies$ $v$ $\lam$-bypasses $B$.
\item\textbf{Reaching neighborhoods:} If $D$ is out-properly labelled then $v$ $3$-reaches $N^+_{X_0}(v)$.
\end{enumerate}
\end{observation}
\begin{proof}
For (i), we can take the family of paths for the definition of $v$ $\lam$-reaching $R$ to be empty. 
For (ii), notice that every subset $R\subseteq B$ has $|R|\leq |B|< \gamma_\lam|D|$ regardless of whether $R$ is reached by $v$ or not. 

For (iii), notice that $\Delta_3\geq k_3$. Let $S$ be a set of $\leq k_3$ labels. Notice that for every $y\in N^+_{X_0\setminus S}(v)\setminus S$, the edge $vy$ is a length $1\leq d_3$ switching path from $v$ to $y$ avoiding $S$. Since $D$ is properly coloured we have $|N^+_{X_0}(v)\setminus (N^+_{X_0\setminus S}(v)\setminus S)|\leq |S|\leq k_3\leq \Delta_3$, and so we have enough paths for the definition of $v$ $\lam$-reaching  $N^+_{X_0}(v)$.
\end{proof}

Observation~\ref{ObservationHereditary} shows that if $R$ is $\lam$-reached by $v$, we can pass to a subset of $R$ and still have it $\lam$-reached.
We will sometimes want to increase the size of a set $R$ and know that it is still reached by $v$. The following lemma shows that we can add the vertex $v$ itself to $R$ and still know that $R\cup\{v\}$ is reached by $v$ \emph{with the same parameter} $\lam$.
\begin{observation}[Reaching one more vertex]\label{ObservationReachingOneMore}
$v$ $\lam$-reaches $R$ $\implies$ $v$ $\lam$-reaches $R\cup\{v\}$.
\end{observation}
\begin{proof}
Let $S$ be a set of $\leq k_{\lam}$ labels.
Recall that $\{v\}$ is a length $0\leq d_\lam$ switching path from $v$ to $v$ avoiding $S$ (using the fact that the first vertex of a path is allowed to be in $S$ in the definition of ``avoids''.) Also, since $v$ $\lam$-reaches $R$,  there are length $\leq d_\lam$ switching paths avoiding $S$ to all except at most $\Delta_\lam$ vertices of $R$. These paths, together with $\{v\}$, give the required paths to show that $v$ $\lam$-reaches $R\cup\{v\}$.
\end{proof}
Consider a set $R_v$ as in the definition of $\lam$-component i.e. $R_v$ is $\lam$-reached by $v$ and $\overline{R_v}$ is $(\lam-3)$-bypassed by $v$. By Observation~\ref{ObservationReachingOneMore} $R_v\cup\{v\}$ is $\lam$-reached by  $v$, and by the monotonicity of bypassing $\overline{R_v\cup \{v\}}$ is $(\lam-3)$-bypassed by $v$. This shows that without affecting anything we could have added the condition ``$v\in R_v$'' to the definition of $\lam$-component.

The next two lemmas show that reaching and bypassing are preserved by unions, as long as we weaken the parameter $\lam$ slightly.
\begin{lemma}[Reaching unions]\label{LemmaUnionReach}
For $m\leq \gamma_{\lam}^{-1}$, suppose that a vertex  $v$ $\lam$-reaches sets $R_1, \dots, R_m\subseteq V(D)$. Then  $v$ $(\lam+3)$-reaches  $\bigcup_{i=1}^m R_i$.
\end{lemma}
\begin{proof}
Let $S$ be a set of $k_{\lam+3}$ labels. Since $k_{\lam+3}\leq k_\lam$ and $v$ $\lam$-reaches $R_i$, there are length $\leq d_\lam$ switching paths avoiding $S$ to all, except possibly $\Delta_\lam$, vertices $x\in R_i$ for each $i$. Therefore there are length $\leq d_\lam\leq d_{\lam+3}$ switching paths avoiding $S$ to all, except possibly $m\Delta_\lam\leq \gamma_\lam^{-1}\Delta_\lam= \twr(\lam+2)\twr(\lam)\leq \twr(\lam+3)=\Delta_{\lam+3}$ vertices in $\bigcup_{i=1}^{m} R_i$.
\end{proof}
%\begin{lemma}[Reaching unions]\label{LemmaUnionReach}
%For $m\leq \gamma_{\lam}^{-1}$, suppose that $v$ $\lam$-reaches sets $R_0, \dots, R_m\subseteq V(D)$. Then  $v$ $(\lam+3)$-reaches  $\bigcup_{i=0}^m R_i$.
%\end{lemma}
%\begin{proof}
%For any set $S$ of $\leq k_{\lam+3}\leq k_\lam$ labels there are length $\leq d_\lam$ switching paths avoiding $S$ to all, except possibly $\Delta_\lam$, vertices $x\in R_i$ for each $i$. Therefore there are length $\leq d_\lam$ switching paths avoiding $S$ to all, except possibly $(m+1)\Delta_\lam\leq (\gamma_\lam^{-1}+1)\Delta_\lam= (\twr(\lam+2)+1)\twr(\lam)\leq \twr(\lam+3)=\Delta_{\lam+3}$ vertices in $R=\bigcup_{i=0}^{m} R_i$.
%\end{proof}

A similar lemma holds for bypassing.
\begin{lemma}[Bypassing unions]\label{LemmaUnionBypass}
For $m\leq \gamma_{\lam-1}^{-1}$, suppose that a vertex $v$ $\lam$-bypasses sets $B_1, \dots, B_m\subseteq V(D)$. Then  $v$ $(\lam-1)$-bypasses  $\bigcup_{i=1}^m B_i$.
\end{lemma}
\begin{proof}
Suppose that $v$ $(\lam-1)$-reaches a set $R\subseteq \bigcup_{i=1}^m B_i$. Without loss of generality we can suppose that $B_1, \dots, B_m$ are ordered so that $|R\cap B_1|\geq |R\cap B_j|$ for $j>1$. Since $R=\bigcup_{i=1}^m R\cap B_i$ we have $|R\cap B_1|\geq |R|/m$. From the monotonicity of reaching, $v$ $\lam$-reaches $R\cap B_1$. Since $v$ $\lam$-bypasses $B_1$, this implies $|R\cap B_1|<\gamma_{\lam}|D|$. This gives $|R|\leq m |R\cap B_1|\leq m\gamma_{\lam}|D|\leq \gamma_{\lam-1}^{-1}\gamma_{\lam}|D|=\twr(\lam+1)\twr(\lam+2)^{-1}|D|\leq \twr(\lam+1)^{-1}|D|=\gamma_{\lam-1} |D|.$ Since  $R$ was arbitrary, we have proved that $v$ $(\lam-1)$-bypasses  $\bigcup_{i=1}^m B_i$.
\end{proof}

Recall that ``two vertices $u$ and $v$ being connected by a path'' is a transitive relation on vertices in an graph. This transitivity is used to show that connected components in a graph are equivalence classes.
The following lemma shows that ``reaching'' also has a kind of transitive property. The lemma plays a similar role in showing that $\lam$-components exist.
\begin{lemma}[Transitivity of reaching]\label{Transitivity}
For a labelled digraph $D$ and  $1\leq \lam\leq \lam_\MAX$, suppose that we have a vertex $v\in V(D)$, and $R$ such that $v$ $\lam$-reaches $R$.
Suppose that we have distinct vertices  $x_0, \dots, x_{\Delta_\lam}\in R$ and a set $W$ such that  $x_i$ $\lam$-reaches $W$ for each $i$.
Then $v$ $(\lam+1)$-reaches $W$.
\end{lemma}
\begin{proof}
Let $S$ be a set of $k_{\lam+1}$ labels. Since $v$ $\lam$-reaches $R$ and $k_{\lam+1}\leq k_\lam$, there is some $i\in 0, \dots, {\Delta_\lam}$ such that there is a  length $\leq d_{\lam}$ switching path $P$ from $v$ to $x_i$ avoiding $S$. Since $x_i$ $\lam$-reaches $W$ and $|S|+|\underline P|\leq k_{\lam+1}+3d_\lam \leq k_\lam$, there is a  length $\leq d_\lam$ switching path $P_w$ avoiding $S$ and   $\underline P$ from $x_i$ to all, except $\Delta_\lam$ vertices of $w\in W$. 

Using Lemma~\ref{LemmaConcatenatePaths}, the paths $P+P_w$ are length $\leq 2d_\lam\leq d_{\lam+1}$ switching paths avoiding $S$ to all except at most $\Delta_\lam\leq\Delta_{\lam+1}$ vertices  $w\in W$. This proves the lemma.
\end{proof}

A consequence of Observation~\ref{ObservationReachBypassExistence} (iii) is that components cannot be much smaller than the neighborhoods of vertices they contain.
\begin{lemma}[Components are larger than neighbourhoods]\label{ComponentLarge} 
Let $D$ be a out-properly labelled, directed graph on $n$ vertices with $X_0$ the set of non-vertex labels in $D$, and $6\leq \lam\leq \lam_\MAX$.
For any $\lam$-component $C$ and  $v\in C$ we have $|C|\geq |N^+_{X_0}(v)|-\delta n-\gamma_{\lam-3}n$.
\end{lemma}
\begin{proof}
By the definition of $C$ being a $\lam$-component, there is a set $R_v$ with $|R_v\setminus C|\leq \delta n$ such that $v$ $(\lam-3)$-bypasses $\overline {R_v}$. By Observation~\ref{ObservationReachBypassExistence} (iii) and the monotonicity of reaching, $v$ $(\lam-3)$-reaches $N^+_{X_0}(v)$. This gives $|N^+_{X_0}(v)\cap \overline {R_v}|\leq\gamma_{\lam-3}n$ which implies the lemma: 
$$|C|\geq |N^+_{X_0}(v)|-|N^+_{X_0}(v)\cap \overline C|\geq  |N^+_{X_0}(v)|-|N^+_{X_0}(v)\cap \overline C\cap \overline{R_v}|-|\overline C\setminus \overline{R_v}|\geq |N^+_{X_0}(v)|-\delta n-\gamma_{\lam-3}n.$$
\end{proof}

\subsection{Constructing $\lam$-components}\label{SectionComponentExistence}
The goal of this section is to show that $\lam$-components exist for suitable $\lam$. The first step towards this is to show that for any vertex $v$ and number $\lam$, there is a set $R_v\subseteq V(D)$ possessing the two properties $R_v$ has in the definition of $\lam$-component.
\begin{lemma}\label{SetReachedBypassed1}
For all vertices $v$ in a labelled digraph $D$ and $1\leq \lam\leq \lam_\MAX$, there is a set $R\subseteq V(D)$ such that $v$ $(\lam+3)$-reaches $R$ and $v$ $\lam$-bypasses $\overline R$.
\end{lemma}
\begin{proof}
We define sets of vertices $R_0, R_1, R_2, \dots, R_m$  recursively as follows.
\begin{itemize}
\item Let $R_0=\emptyset.$
\item For each $i\geq 1$, if possible, choose $R_{i}$ to be any set disjoint from $R_0\cup\dots\cup R_{i-1}$ which is $\lam$-reached by $v$, and also $|R_{i}|\geq \gamma_{\lam}|D|$.
\item Otherwise, if no such $R_{i}$ exists, we stop with $m=i-1$.
\end{itemize}
Notice that the sets $R_1, \dots, R_m$ are all disjoint and satisfy $|R_{i}|\geq \gamma_{\lam}|D|$ which implies that $m\leq \gamma_\lam^{-1}$.
Set $R=R_1\cup\dots\cup R_m$.

By definition of $m$, $v$ $\lam$-bypasses $\overline R$---indeed otherwise we could choose a set $R_{m+1}$ of size $\gamma_{\lam}|D|$ disjoint from $R$ which is $\lam$-reached by $v$, contradicting the fact that we stopped at $m$.  
By Lemma~\ref{LemmaUnionReach}   $v$  $(\lam+3)$-reaches $R$.
This completes the proof.
\end{proof}

Notice that in the above lemma would be stronger if it produced a set $R$ with $R$ $\lam$-reached by $v$ and $\overline R$ $\lam'$-bypassed by $v$ for $\lam'>\lam$ (rather than $\lam'<\lam$ as Lemma~\ref{SetReachedBypassed1} gives us.) The next lemma tries to prove something like this---it produces two sets $R$ and $B$ which are ``nearly complementary'' such that $v$ $\lam$-reaches $R$ and $\lam'$-bypasses $B$ for $\lam'>\lam$.
\begin{lemma}\label{SetReachedBypassed2}
Let $D$ be a labelled digraph on $n$ vertices, $\lam_0\in \mathbb{N}$ with  $43\delta^{-1} \leq \lam_0\leq \lam_\MAX$, and $v\in V(D)$.  There are two sets of vertices $R$ and $B$ satisfying the following.
\begin{enumerate}[(i)]
\item $|V(D)\setminus (R\cup B)|\leq \delta n/3$.
\item There is a $\lam$ with $\lam_0-42\delta^{-1}\leq \lam\leq \lam_0$ such that
\begin{itemize}
\item $v$ $(\lam-4)$-reaches $R$.
\item $v$ $\lam$-bypasses $B$.
\end{itemize}
\end{enumerate}
\end{lemma}
\begin{proof}
Define $\lam_1, \dots, \lam_{6\delta^{-1}}$, $R_1, \dots, R_{6\delta^{-1}}$ as follows.
\begin{itemize}
\item For each $i$, set $\lam_i=\lam_{i-1}-7$.
\item Let $R_i$ be a set which is $(\lam_i+3)$-reached by $v$ and with $\overline R_i$  $\lam_i$-bypassed by $v$. Such a set exists by Lemma~\ref{SetReachedBypassed1}.
\end{itemize}
We show that there is some index $m$ satisfying a property like part (i) of the lemma.
\begin{claim}
There is some $m\in \{1, \dots, 6\delta^{-1}\}$ for which $|V(D)\setminus (R_m\cup \overline{R_{m-1}})|\leq \delta n/3$.
\end{claim}
\begin{proof}
Suppose for the sake of contradiction that $|V(D)\setminus (R_i\cup \overline{R_{i-1}})|> \delta n/3$ for all $i=1, \dots, 6\delta^{-1}$.
Notice that we have  $|R_i\cap \overline{R_{i-1}}|\leq \gamma_{\lam_{i-1}}n$  (since $v$ $\lam_{i-1}$-bypasses $\overline{R_{i-1}}$, $v$ $(\lam_i+3)$-reaches $R_i\cap \overline{R_{i-1}}$, and $\lam_i+3\leq \lam_{i-1}$.) We also have $|R_{i-1}\setminus R_i|=|V(D)\setminus (R_i\cup \overline{R_{i-1}})|> \delta n/3$ for $i<6\delta^{-1}$. Combining these, we get the following
$$|R_{i}|=|R_{i}\cap R_{i-1}|+|R_i\cap \overline{R_{i-1}}|
=|R_{i-1}|-|R_{i-1}\setminus R_i| + |R_i\cap \overline{R_{i-1}}|
< |R_{i-1}|-(\delta/3-\gamma_{\lam_{i-1}})n.$$
Notice that for all $i\leq 6\delta^{-1}$ we have $\lam_i\geq 1$, and so $\gamma_{\lam_{i}}\leq \lam_{\MAX}^{-1}\leq \delta/12$ which implies $|R_{i}|< |R_{i-1}|-\delta n/4$.
Iterating this gives $0\leq |R_i|< |R_1|-(i-1)\delta n/4\leq n-(i-1)\delta n/4$. 
This is a contradiction for $i= 6\delta^{-1}$.
\end{proof}

Set $R=R_m$, $B=\overline{R_{m-1}}$ and $\lam=\lam_{m-1}$. Then $v$ $(\lam-4)$-reaches $R_i$ and $\lam$-bypasses $B$ by the constructions of $R_{i}$ and $R_{i-1}$. We have $|V(D)\setminus (R\cup B)|\leq \delta n/3$ by choice of $m$. Finally we have $\lam_0\geq \lam \geq \lam_0-7m\geq \lam_0-42\delta^{-1}$.
\end{proof}

As a prelude to constructing components we give a condition under which a singleton $\{v\}$ is a $\lam$-component.
\begin{lemma}\label{SingleVertexComponent}
Let $D$ be a labelled digraph on $n\geq N_0$ vertices and $4\leq \lam\leq \lam_\MAX$.
Suppose that $v$ $\lam$-bypasses $B$ with $|\overline B|\leq \delta n/2$. Then  $\{v\}$ is a $\lam$-component.
\end{lemma}
\begin{proof}
To prove the lemma we need to choose a set $R_v$ and show that it satisfies all the properties of the set $R_v$ in the definition of $\lam$-component.
Apply Lemma~\ref{SetReachedBypassed1} to get a set $R_v$ which is $\lam$-reached by $v$ and with $\overline{R_v}$ $(\lam-3)$-bypassed by $v$. 
Notice that since $v$ $\lam$-bypasses $B$, we must have $|R_{v}\cap B|\leq \gamma_{\lam} n$. This gives $|R_{v}\symdiff \{v\}|\leq |R_{v}|+1\leq |R_{v}\cap B|+|\overline B|+1\leq \gamma_{\lam}n+\delta n/2+1\leq \delta n$. 
\end{proof}

The following lemma is a purely technical tool which we will need. A $r$-uniform multihypergraph $\mathcal H$ with $n$ vertices and $m$ edges is a family of $m$ size $r$ subsets of $[n]$ with the possibility of $\mathcal H$ containing several copies of the same subset.
\begin{lemma}\label{LargeIntersectionInHypergraph}
Let $\mathcal H$ be a $\gamma n$-uniform multihypergraph with $n$ vertices and $m$ edges. Then, for any $t$ with $\gamma/2\geq 2t/m$, there are $t$ edges $T_1, \dots, T_t\in \mathcal H$ with 
$|T_1\cap \dots \cap T_t|\geq {\left(\frac{\gamma}{2}\right)}^{t}n.$
\end{lemma}
\begin{proof}
 Let $\mathcal T$ be a set of $t$ distinct edges  of $\mathcal H$ chosen uniformly at random from all such sets. To prove the lemma it is sufficient to show that the expected size of the intersection of the edges in $\mathcal T$ is at least ${\left(\frac{\gamma}{2}\right)}^{t}n.$
For any vertex $v\in V(\mathcal{H})$, let $d(v)$ be the number of edges of $\mathcal H$ containing $v$. 
Let $V_{\geq t}=\{v\in V(\mathcal H): d(v)\geq t\}$.
By linearity of expectation we have the following.
\begin{align*}
\mathbb{E}\left(\left|\bigcap_{E\in \mathcal T} E\right|\right)
&= \sum_{v\in V(\mathcal{H})}\mathbb{P}\left(v\in \bigcap_{E\in \mathcal T} E \right)
= \sum_{v\in V_{\geq t}}\mathbb{P}\left(v\in \bigcap_{E\in \mathcal T} E \right)\\
&=\sum_{v\in  V_{\geq t}}\frac{\binom{d(v)}{t}}{\binom{m}{t}}
\geq \sum_{v\in V_{\geq t}} \left(\frac{(d(v)-t)}{m}\right)^t.
\end{align*}
%The last inequality comes from  ``$(d-i)/(m-i)\geq (d-t)/m$ for $i\leq t$.''
The inequality comes from ``${\binom{d}{t}}\big/{\binom{m}{t}}\geq \left({(d-t)}/{m}\right)^t$ for $d\geq t$.'' 
Using convexity of $f(x)=x^t$ we can prove the lemma.
$$\mathbb{E}\left(\left|\bigcap_{E\in \mathcal T} E\right|\right)
\geq \sum_{v\in V_{\geq t}} \left(\frac{(d(v)-t)}{m}\right)^t
\geq \left(\sum_{v\in V_{\geq t}} \frac{(d(v)-t)}{nm}\right)^t n
\geq\left(\frac{\gamma m -2t}{m}\right)^t n
\geq \left(\frac{\gamma}{2}\right)^t n.$$
The third inequality comes from $\sum_{v\in V(G)} d(v)=\gamma n m$ and $\sum_{v\in V(G)\setminus V_{\geq t}} d(v)\leq tn$.
The last inequality comes from  $\gamma/2\geq 2t/m$.
\end{proof}

The following lemma is the main result of this section. It implies that for a given $\lam_0$, there is a $\lam$-component $C$ for some $\lam$ which is close to $\lam_0$. In addition the lemma gives some control over where the component $C$ is located---given any set $B_0$ which is $\lam_0$-bypassed, we can choose $C$ to be outside $B_0$.
\begin{lemma}\label{ComponentExistence}
Let $D$ be a labelled digraph on $n\geq N_0$ vertices and $87\delta^{-2} \leq \lam_0\leq \lam_\MAX$.
Suppose we have $v_0\in V(D)$ and $B_0\subseteq V(D)$ such that $v_0\in \overline{B_0}$ and  $v_0$ $\lam_0$-bypasses $B_0$. 
Then there is a nonempty $C\subseteq \overline{B_0}$ such that $C$ is a $\lam$-component with $\lam_0-87\delta^{-2}\leq \lam\leq \lam_0$. %\textbf{Proof also gives that $v_0$ $(\lam_0, r)$-reaches $C$. Do we want this or should we delete it from the proof.}
\end{lemma}

\begin{proof}
%Apply Lemma~\ref{SetReachedBypassed1} to find a set $R_0$  containing $v_0$ such that $v_0$ $(\lam_0,r)$-reaches $R_0$ and $(\lam_0-4,r)$-bypasses $\overline{R_0}$. 
We start with the following claim.
\begin{claim}\label{MaximalBypass}
There is a vertex $v'\in \overline {B_0}$, $\lam'\in [\lam_0-86\delta^{-2},\lam_0]$, and a set $B'\supseteq B_0$ with the following properties.
\begin{itemize}
\item $v'$ $\lam'$-bypasses $B'$. 
\item For every $u \in \overline{B'}$, if there is a set $B_u\supset B'$ such that $u$ $(\lam'-43\delta^{-1})$-bypasses $B_u$, then $|B_u|<|B'|+\delta n/2$.
\end{itemize} 
\end{claim}
\begin{proof}
%Let $v_1=v_0$, $B_1=B_0\cup \overline{R_0}$ and $\lam_1=\lam_0-5$. 
%By Lemma~\ref{UnionBypass}, $v_1$ $(\lam_1,r)$-bypasses ${B_1}$. %By Lemma~\ref{UnionBypass}, $v_0$ $(\lam_0-5,r)$-bypasses ${B_1}$.  
Using $B_0$, $v_0$, and $\lam_0$ from the lemma, we define $B_1, \dots, B_m$, $v_1, \dots, v_m$, and $\lam_1, \dots, \lam_m$ as follows.
\begin{itemize}
\item For each $i$, set $\lam_{i+1}=\lam_i-43\delta^{-1}$. 
\item For each $i$, if possible, choose a vertex $v_{i+1}\in\overline{B_i}$ and a set $B_{i+1}\supset B_i$  such that  $v_{i+1}$ $\lam_{i+1}$-bypasses $B_{i+1}$ and $|B_{i+1}|\geq |B_i|+\delta n/2$.
\item Otherwise, if no such pair of $v_{i+1}$ and $B_{i+1}$ exists, then stop with $m=i$.
\end{itemize}
Notice that since  $|B_{i+1}|\geq |B_i|+\delta n/2$ for $i<m$, we stop with $m\leq 2\delta^{-1}$. 
Let $\lam'=\lam_m$, $v'=v_m$, and $B'=B_m$.
Since $m\leq 2\delta^{-1}$, we have $\lam'= \lam_0-43\delta^{-1}m\geq \lam_0-86\delta^{-2}$.
We have $B'=B_m\supseteq B_{m-1}\supseteq \dots\supseteq B_0$ and $v'\in \overline{B_{m-1}}\subseteq \overline{B_0}$ as required.
The vertex $v'$ $\lam'$-bypasses $B'$ by choice of $v_m$ and $B_m$.
The fact that ``for every $u \in \overline{B'}$, if there is a set $B_u\supset B'$ such that $u$ $(\lam'-43\delta^{-1})$-bypasses $B_u$, then $|B_u|< |B|+\delta n/2$'' is equivalent to the fact that we stopped at $m$. 
\end{proof}

Apply Lemma~\ref{SetReachedBypassed2} to $v'$ and $\lam'$ in order to obtain sets $R$ and $B$ and $\lam''\in [\lam'-42\delta,\lam']$ such that $|V(D)\setminus(R\cup B)|\leq \delta n/3$, $v'$ $(\lam''-4)$-reaches $R$, and $v'$ $\lam''$-bypasses $B$.
By Lemma~\ref{LemmaUnionBypass} and the monotonicity of bypassing, $v'$  $(\lam''-1)$-bypasses $B\cup B'$.

We make the following definition
$$S=\{x\in  R: x \ (\lam''-2)\text{-reaches some } T \subseteq B\cup B' \text{ with } |T|\geq \gamma_{\lam''-3} n\}.$$

 From the definition of $S$ and the monotonicity of reaching, we have that for every $v\in R\setminus S$ the vertex $v$ $(\lam'' -3)$-bypasses $B\cup B'$.
Using Lemma~\ref{LargeIntersectionInHypergraph} and the ``transitivity of reaching'' we show that $S$ is small.
\begin{claim}
$|S|\leq 4\Delta_{\lam''}\gamma_{\lam''}^{-1}$.
\end{claim}
\begin{proof}
Suppose for the sake of contradiction that $|S|> 4\Delta_{\lam''}^2\gamma_{\lam''}^{-1}$.
For each $s\in S$, choose some set $T_s\subseteq B\cup B'$ with $|T_s|=\gamma_{\lam''-3} n$ which is $(\lam''-2)$-reached by $s$.
Let $\mathcal H=\{T_s:s\in S\}$. Notice that $\mathcal H$ is an $(\gamma_{\lam''-3}n)$-uniform multihypergraph with $|S|$ edges. 
Notice that  $|S|> 4\Delta_{\lam''}\gamma_{\lam''}^{-1}$ implies $\gamma_{\lam''-3}/2\geq 2(\Delta_{\lam''-2}+1)/|S|$. Therefore we can apply Lemma~\ref{LargeIntersectionInHypergraph} to $\mathcal H$ with $t=\Delta_{\lam''-2}+1$ and $\gamma=\gamma_{\lam''-3}$ in order to find $\Delta_{\lam''-2}+1$ sets $T'_0, \dots, T'_{\Delta_{\lam''-2}}\in \mathcal H$ with 
$$\left|\bigcap_{i=0}^{\Delta_{\lam''-2}} T'_i\right|\geq \left(\frac{\gamma_{\lam''-3}}2\right)^{\Delta_{\lam''-2}+1} n>\gamma_{\lam''-1}n.$$
The second inequality comes from $\gamma_{\lam''-1}^{-1}=\twr(\lam''+1)\geq (2\twr(\lam''-1))^{\twr(\lam''-2)+1}=(2/\gamma_{\lam''-3})^{{\Delta_{\lam''-2}+1}}$.
By Lemma~\ref{Transitivity} applied with $\lam=\lam''-2$, $v=v'$, $R=R$, $W=\bigcap_{i=0}^{\Delta_{\lam''-2}} T'_i$, and $x_i$ the vertex of $S$ which $(\lam''-2)$-reaches $T'_i$,  we get that  $v'$ $(\lam''-1)$-reaches $\bigcap_{i=0}^{\Delta_{\lam''-2}} T'_i$. This contradicts $v'$ $(\lam''-1)$-bypassing $B\cup B'$.
\end{proof}

Let $C= R\setminus (S\cup B'\cup B)$ and $\lam=\lam''-2$. 
%We will show that $C$ and $\lambda$ satisfy the conditions of the lemma.
 %\textbf{delete: Since $C\subseteq R_0$ we have that $v_0$ $(\lam_0,r)$-reaches $C$.} 
 Since $C\subseteq \overline {B'}$ and $B'\supseteq B_0$ we have $C\subseteq  \overline{B_0}$. From the definitions of $\lam'$ and $\lam''$ we have that $\lam_0-87\delta^{-2}\leq \lam\leq \lam_0$.
 \begin{claim}
$C$ is a $\lam$-component.
 \end{claim}
 \begin{proof}
For each $v\in C$ apply Lemma~\ref{SetReachedBypassed1} to get a set $R_v$ which is $\lam$-reached by $v$ and with $\overline{R_v}$ $(\lam-3)$-bypassed by $v$. To prove the claim, it is enough to show that $|R_v\symdiff C|\leq \delta n$. We'll do this by showing $|R_v\setminus C|\leq \delta n/2$ and $|C\setminus R_v|\leq \delta n/2$.

First we show that  $|R_v\setminus C|\leq \delta n/2$.
Notice that $C=R\cap \overline{S}\cap \overline {B'}\cap \overline B=(R\cup B)\cap \overline{S}\cap \overline {B'}\cap \overline B$.
Notice that since $v\in R\setminus S$ and $\lam=\lam''-2$ we have $|R_v\cap (B\cup B')|\leq\gamma_{\lam''-3} n$. 
Combining these with $|S|\leq 4\Delta_{\lam''}\gamma_{\lam''}^{-1}$ and $|V(D)\setminus(R\cup B)|\leq \delta n/3$ we get
\begin{align*}
|R_v\setminus C|&=|R_v\setminus ((R\cup B)\cap \overline{S}\cap \overline {B'}\cap \overline B)|\\
&=|R_v\cap ((\overline{R \cup B})\cup S\cup B'\cup B)|\\
&\leq|R_v\cap (\overline{R \cup B})|+|R_v\cap S|+|R_v\cap (B\cup B')|\\
&\leq |V(D)\setminus (R\cup B)|+|S|+|R_v\cap (B\cup B')|\\
&\leq \delta n/3+4\Delta_{\lam''}\gamma_{\lam''}^{-1}+\gamma_{\lam''-3} n\leq \delta n/2. 
\end{align*}
The last inequality comes from $n\geq N_0$ and $\lam \leq \lam_\MAX$.

Next we show that $|C\setminus R_v|\leq \delta n/2$.
Using  $v\in R\setminus S$ and the monotonicity of bypassing we get that $v$ $(\lam-3)$-bypasses ${B'}$.
By Lemma~\ref{LemmaUnionBypass}, $v$  $(\lam-4)$-bypasses $\overline{R_v}\cup B'$.
Now using the monotonicity of bypassing we have a vertex $v\in \overline {B'}$ and a set $\overline {R_v}\cup B'\supseteq B'$ such that $v$ $(\lam'-43\delta^{-1})$-bypasses $\overline {R_v}\cup B'$.
From Claim~\ref{MaximalBypass} we have $|\overline{R_v}\cup B'|\leq |B'|+\delta n/2$  which implies $|\overline{R_v}\setminus  B'|\leq \delta n/2$.
This gives us 
$$|C\setminus R_v|= |C\cap \overline{R_v}|\leq|\overline{R_v}\setminus B'|+|C\cap B'|= |\overline{R_v}\setminus B'|\leq \delta n/2.$$
\end{proof}
We have now proved that $C$ satisfies all the requirements of the lemma aside from ``$C$ is nonempty''. Thus, for the remainder of the proof we can assume that $C$ is empty, or equivalently $R\subseteq  S\cup B'\cup B$. We'll show that $\{v'\}$ is a $\lam$-component  satisfying the conditions of the lemma. Notice that $\{v'\}\subseteq \overline{B_0}$ holds as a consequence of the definition of $v'$ in Claim~\ref{MaximalBypass}.

By $|S|\leq 4\Delta_{\lam''}\gamma_{\lam''}^{-1}\leq \gamma_{\lam+1}n$ and Observation~\ref{ObservationReachBypassExistence} (ii), $v$ $(\lam+1)$-bypasses $S$. By Lemma~\ref{LemmaUnionBypass}, $v'$ $\lam$-bypasses $S\cup B'\cup B$.  From  $R\subseteq  S\cup B'\cup B$ we obtain $|V(D)\setminus (S\cup B'\cup B)|=|V(D)\setminus (R\cup S\cup B'\cup B)|\leq |V(D)\setminus (R\cup B)|\leq \delta n/3$. 
By Lemma~\ref{SingleVertexComponent} applied with $v=v'$, $\lam=\lam$,  and $\hat B=S\cup B'\cup B$, we have that $C'=\{v'\}$ is a $\lam$-component which satisfies the conditions of the lemma.
%If $C$ is empty, then let $C'=\{v'\}$. Let $R_{v'}$ be the set produced by Lemma~\ref{SetReachedBypassed1} applied to $v$ and $\lam-3$. Notice that since $v_{i}$ $(\lam''-7, i)$-bypasses $B'\cup B\cup S$, we must have $|R_{v'}\cap (B'\cup B\cup S)|\leq \gamma_{\lam''?}$. This gives $|R_{v'}|\leq |R_{v'}\cap (B'\cup B\cup S)|+|D\setminus (B'\cup B\cup S)|\leq \gamma_{\lam''?}+| R\setminus (S\cup B'\cup B){???},|+|D\setminus (R\cup B)|\leq 2\delta n$. Therefore $C'$ is a (???)-component, proving the result.
\end{proof}

By iteratively applying Lemma~\ref{ComponentExistence} we can find a sequence of $\lam$-components for decreasing $\lam$. We'll use this sequence in the proof of Theorem~\ref{TheoremDigraph}.
\begin{lemma}\label{ComponentTower}
Fix $m=4\delta^{-1}$ and let $D$ be a  labelled digraph on $n \geq N_0$ vertices.
For any $v\in V(D)$  we can choose $C_1, \dots, C_m$ and $ \lam_1, \dots,  \lam_m$ such that for each $i$, $C_i$ is a $\lam_i$-component with  $|C_i\setminus C_{i+1}|\leq \delta n$ and $1\leq \lam_{i+1}-88\delta^{-2}\leq \lam_{i}\leq \lam_{i+1}-9\leq \lam_\MAX$ for $i=1, \dots, m$.

In addition there is a length $\leq d_{\lam_\MAX/2}$ switching path which starts at $v$ and passes through all of $C_1, \dots, C_m$.
\end{lemma}
\begin{proof}
We will choose vertices $v_{1}, \dots, v_{m+1}$, paths $P_{1}, \dots, P_{m+1}$, sets $R_{1}, \dots, R_{m+1}$, numbers $\lam_{1}, \dots, \lam_{m+1}$, and components $C_1, \dots, C_m$. 
They will have the following properties.
\begin{enumerate}[(i)]
\item For $i\leq m$, $1\leq \lam_{i+1}-88\delta^{-2}\leq \lam_{i}\leq \lam_{i+1}-9\leq\lam_\MAX/4$.
\item For $i\leq m+1$, $v_i$ $\lam_i$-reaches $R_i$ and  $v_i$ $(\lam_i-3)$-bypasses $\overline R_i$.
\item For $i\leq m+1$, $P_i$ is a switching path from $v$ to $v_i$ passing through $v_{i+1}, \dots, v_{m}$. Also $|P_i|\leq d_{\lam_{i+1}}+d_{\lam_{i+2}}+\dots +d_{\lam_{m}}$.
\item For $i\leq m$, $C_i$ is a $\lam_i$-component with  $C_i\subseteq R_{i+1}$, $v_i\in C_i$,  and $|C_i\symdiff R_i|\leq \delta n$.
\end{enumerate}
Once we have constructed these sequences, then it is easy to see that the components  $C_1, \dots, C_m$, the numbers $\lam_1\dots, \lam_m$, and the path $P_1$ satisfy the conditions of the lemma. Indeed $C_i$ is a $\lam_i$-component with  $|C_i\setminus C_{i+1}|\leq |R_{i+1}\setminus C_{i+1}|\leq |R_{i+1}\symdiff C_{i+1}|\leq \delta n$ by (iv). By (i), (iii), and (iv) the path $P_1$ is a length $\leq m d_{\lam_m}\leq md_{(\lam_\MAX/4)}\leq d_{\lam_\MAX/2}$ switching path starting from $v$ which passes through the vertex $v_i\in C_i$ for $i=1, \dots, m$.

We will construct $v_i, P_i, R_i, \lam_i$, and $C_i$ in reverse order starting with $i=m+1$ and ending with $i=1$.
Let $v_{m+1}=v$, $P_{m+1}=\{v\}$, and $\lam_{m+1}=\lam_\MAX/4$. Use Lemma~\ref{SetReachedBypassed1} to find a set $R_{m+1}$  such that $v_{m+1}$ $(\lam_{m+1})$-reaches $R_{m+1}$ and $(\lam_{m+1}-3)$-bypasses $\overline {R_{m+1}}$. By construction, conditions (i) -- (iii) are satisfied for $v_{m+1}, P_{m+1}, R_{m+1}$, and $\lam_{m+1}$. Condition (iv) doesn't need to be checked since we do not have a component $C_{m+1}$.

For each $i\leq m+1$, suppose that we have constructed $v_{i}, P_{i}, R_{i}$, and $\lam_{i}$. We build $v_{i-1}, P_{i-1}, R_{i-1}, \lam_{i-1}$, and $C_{i-1}$ as follows.

%Suppose that $|R_i|>\delta n$.
By the monotonicity of bypassing and Observation~\ref{ObservationReachingOneMore}, $v_i$ $\lam_i$-reaches $R_i\cup \{v_i\}$ and $(\lam_i-3)$-bypasses $\overline{R_i\cup \{v_i\}}$.
Since  $v_i$ $\lam_i$-reaches $R_i\cup \{v_i\}$ and $|\underline{P_i}|\leq 3id_{\lam_{m+1}}\leq k_{\lam_i}$, there is a subset $R'_i\subseteq R_i\cup \{v_i\}$ of order at least $|R_i|-\Delta_{\lam_i}$ such that there are length $\leq d_{\lam_i}$ switching paths from $v_i$ to all $r\in R'_i$ avoiding $\underline{P_{i}}$. Without loss of generality we can assume $v_i\in R_i'$ (since there is a length $0$ switching path from $v_i$ to $v_i$ avoiding any  set of labels.)
%By Observation~\ref{ObservationHereditary} $v_i$ $\lam_i$-reaches $R'_i$. 
By $\Delta_{\lam_i}\leq \gamma_{\lam_i-3}n$, Observation~\ref{ObservationReachBypassExistence} (ii), and Lemma~\ref{LemmaUnionBypass}, $v_i$ $(\lam_i-4)$-bypasses $\overline {R'_i}$. 
Apply Lemma~\ref{ComponentExistence} with $v_0=v_{i}$, $\lam_0=\lam_{i}-9$, and $B_0=\overline{R'_i}$ in order to find a nonempty $\lam$-component $C_{i-1}$ contained in $R'_i$ for some $\lam_i-9-87\delta^{-2}\leq\lam\leq \lam_i-9$. Let $\lam_{i-1}=\lam$.
Let $v_{i-1}$ be any vertex in $C_{i-1}$.
Since  $v_{i-1}\in C_{i-1}\subseteq R'_i$, there is a length $\leq d_{\lam_{i}}$ switching path $Q$ from $v_{i}$ to  $v_{i-1}$ which avoids $\underline{P_{i}}$. Let $P_{i-1}=P_{i}+Q$.
Let   $R_{{i-1}}$ be the set from the definition of ``$\lam_{i-1}$-component'' such that $|R_{i-1}\symdiff C|\leq \delta n$, $v_{i-1}$ $\lam_{i-1}$-reaches $R_{i-1}$, and $(\lam_{i-1}-3)$-bypasses $\overline{R_{i-1}}$. 

We claim that $v_{i-1}, P_{i-1}, R_{i-1}, \lam_{i-1}$, and $C_{i-1}$ satisfy properties (i) -- (iv). For property  (i) we
have $\lam_i-88\delta^{-2}\leq \lam_i-9-87\delta^{-2}\leq \lam_{i-1}\leq \lam_{i}-9$ and also  $\lam_{i-1}\geq \lam_{m+1}-(m-i+1)88\delta^{-2}\geq \lam_\MAX-252\delta^{-3}\geq 1$.
 Property (ii) holds since we chose $R_{i-1}$ so that $v_{i-1}$ $\lam_{i-1}$-reaches $R_{i-1}$ and $(\lam_{i-1}-3)$-bypasses $\overline{R_{i-1}}$.
To see that property (iii) holds, notice that it held for $P_{i}$, and that $Q$ is a length $\leq d_{\lam_{i-1}}$ switching path  avoiding $\underline{P_i}$  from $v_i$ to  $v_{i-1}$. Using Lemma~\ref{LemmaConcatenatePaths} we have that $P_{i-1}=P_i+Q$ is a length $e(P_i)+e(Q)\leq d_{\lam_{i-1}}+\dots+d_{\lam_m}$ switching path from $v$ to $v_{i-1}$.
For property (iv), $C_{i-1}$ being a $\lam_{i-1}$-component with  $C_{i-1}\subseteq R_{i}$ comes from our application of Lemma~\ref{ComponentExistence}, $v_{i-1}\in C_{i-1}$ comes from our choice of $v_{i-1}$, and $|C_{i-1}\symdiff R_{i-1}|\leq \delta n$ comes from the  choice of $R_{i-1}$.
\end{proof}

\subsection{Growth of $\lam$-components}\label{SectionComponentGrowth}
Notice that so far in Section~\ref{SectionDirectedGraphs}, ``amidstness'' has only come up in the statement of Theorem~\ref{TheoremDigraph}.  In this section we build a link between amidstness and $\lam$-components. 
First we will need a more precise notion of amidstness which incorporates the parameter $\lam$.
\begin{definition}
Let $u$ and $v$ be two vertices in an edge-labelled, directed graph $D$, $c$ a label, and $S$ a set of labels. We say that $c$ is $(\lam,  S)$-\emph{amidst} $u$ and $v$ if $c\not\in S$ and there is a length $\leq 2d_\lam$ switching path $P=(u,p_1,\dots, p_{d}, v)$ avoiding $S$ from $u$ to $v$ such that the following hold.
\begin{enumerate}[(i)]
\item There are no edges of $P$ labelled by $c$.
\item If $c$ is a vertex of $D$ then  $c$ is in $\{p_1, \dots, p_d, v\}$.
\end{enumerate}
\end{definition}

The following is an extension of Lemma~\ref{LemmaConcatenatePaths}. It shows that when we concatenate two switching paths, then the vertex at which we concatenate automatically becomes amidst the endpoints of the concatenated path.
\begin{lemma}\label{LemmaAmidstConcatenate}
In a labelled digraph $D$, let $S$ be a set of vertices,  $P$ a length $\geq 1$ and  $\leq d_\lam$ switching path from $u$ to $x$ which avoids $S$, and $Q$ a length $\leq d_\lam$ switching path from $x$ to $v$ which avoids $S$ and ${\underline P}$. Then $P+Q$ witnesses $x$ being $(\lam, S)$-amidst $u$ and $v$.
\end{lemma}
\begin{proof}
Using Lemma~\ref{LemmaConcatenatePaths} we have that $P+Q$ is a length $\leq 2d_\lam$ switching path avoiding $S$ from $u$ to $v$ which passes through $x$. 
We have $x \not\in S$ since $P$ avoids $S$, $P$ ends with $x$, and $|P|\geq 2$.
Part (ii) of the definition of $(\lam, S)$-amidst holds since $x\in P+Q$ and $x$ is not the starting vertex of $P$.
It remains to show that part (i) of the definition of $(\lam, S)$-amidst holds. 
Note that $P$ has no edges labelled by $x$ since $P$ is a switching path ending with $x$, and  $Q$ has no edges labelled by $x$ since $Q$ avoids ${\underline P}\ni x$. These imply that $P+Q$ has no edges labelled by $x$.
\end{proof}

Recall that the idea behind ``amidstness'' was to identify labels which can be used to extend switching paths. The next lemma makes this precise.
\begin{lemma}\label{LemmaAmidstAddOneVertex}
Let $x$ be a label in a labelled digraph $D$ which is $(\lam, S)$-amidst $v$ and $y$. Suppose that we have a vertex $z$ with $z\not\in S$ and the edge $yz$ present and labelled by $x$.
Then  there is a length $\leq 2d_\lam+1$ switching path $P$ from $v$ to $z$ avoiding $S$.
\end{lemma}
\begin{proof}
%Notice that for every $z\in R\setminus R_S$, there is a pair $(x,y)\in Z_{S\cup {\underline P}}^L$ with $yz$ present and labelled by $x$. In particular, this means that $x$ is $(\lam, S)$-amidst $v$ and $y$, which gives 
Since $x$ is $(\lam, S)$-amidst $u$ and $v$, there is a  length $\leq {2d_{\lam}}$ switching path $Q$ from $v$ to $y$ avoiding $S$ and having no edges labelled by $x$. In addition if $x$ is a vertex then  $x\in V(Q)\setminus \{v\}$.
If $z\in Q$, then we are done by choosing $P$ to be the subpath of $Q$ ending with $z$.
Otherwise we take $P=Q+z$ to get a path from $v$ to $z$.
To see that this is a switching path first notice that  the label of the last edge $yz$ is $x$ which is not present on the edges of $P$. In addition if $x$ is a vertex then $x\in V(Q)\setminus \{v\}= V(P)\setminus \{v,z\}$. Thus $P$ is a switching path. 
Notice that $x\not\in S$ since $x$ is $(\lam, S)$-amidst $v$ and $y$. Combining this with $z\not\in S$ and the fact that $Q$ avoided $S$ shows that $P$ avoids $S$.
Thus $P$ is a length $\leq 2d_\lam+1$ switching path from $v$ to $z$ avoiding $S$ as required.
\end{proof}

Recall that for a label $x$ and a vertex $v$,  $N^+_{\{x\}}(v)$ denotes the set of vertices $y$ with $vy$ an edge labelled by $x$. 
In a out-properly labelled digraph we always have $|N^+_{\{x\}}(v)|\in \{0,1\}$.
For a labelled digraph $D$, let 
$$L^+(D)=\{(x,y): \text{$x$ is a label and $y\in V(D)$ with $N^+_{\{x\}}(y)\neq  \emptyset$}\}.$$
Equivalently we have that $L^+(D)$ is the set of pairs $(x,y)$ with $x$ a label and $y\in V(D)$ such that there is an edge in $D$ starting at $y$ labelled by $x$.

The following lemma gives a connection between amidstness and $\lam$-components. It shows that for a $\lam$-component $C$, most triples of the form $(v,x,y)\subseteq C\times (X_0\cup C)\times C$ have $x$ amidst $v$ and $y$.
\begin{lemma}\label{MengerLemma}
Let   $D$ be an out-properly labelled digraph with $|D|=n\geq N_0$, $\lam\in \mathbb N$ with $3\leq \lam \leq \lam_\MAX$, $X_0$ the set of non-vertex labels of $D$, and $C$ a $\lam$-component. Fix a vertex $v \in C$. 
To every set of labels $S$ with $|S|\leq k_\lam-3d_\lam$ we can assign sets $C_S\subseteq C$ and $Z_S\subseteq (X_0\cup C_S)\times C_S$ with the following properties.
\begin{enumerate}[(i)]
\item For every $(x,y)\in Z_S$, the label $x$ is $(\lam, S)$-amidst $v$ and $y$.
\item $|C_S|\geq |C|-2\delta n$ and $|Z_S|\geq |(X_0\cup C_S)\times C_S|-4\delta n|C_S|$.
\item For two sets $S$ and $T$ we have $|L^+(D)\cap (Z_S\setminus Z_T)|\leq 6\Delta_\lam n$.
\end{enumerate}
\end{lemma}
\begin{proof}
Since $C$ is a $\lam$-component, for every $u\in C$ there is a set $R^u\subseteq C$ with $|R^u|\geq |C|-\delta n$ so that $u$ $\lam$-reaches $R^u$. 
For a set of labels $S$, let $R^u_S$ be the set of $y\in R^u$ for which there is a length $\leq d_\lam$ switching path from $u$ to $y$ avoiding $S$. For each $u$, $S$ and $y\in R^u_S$, fix such a switching path $P^{u,y}_S$.
Since $u$ $\lam$-reaches $R^u$, we have 
\begin{equation}\label{EqRuRuSDifference}
|R^u\setminus R^u_S|\leq \Delta_\lam \text{ for all $S$ with $|S|\leq k_\lam$}. 
\end{equation}
For a pair of sets $S$ and $T$ with $|S|, |T|\leq k_\lam$, we trivially have $|R^u_S\setminus R^u_T|\leq |R^u\setminus R^u_T|$ which implies 
$|R^u_S\setminus R^u_T|\leq \Delta_\lambda.$
We now define the sets $C_S$ and $Z_S.$
\begin{align*}
C_S&=R^v_S\setminus \{v\},\\
Z_S^0&=\{(x,y)\in X_0\times C_S: P^{v,y}_S \text{ avoids $x$ and $x\not\in S$}\},\\
Z_S^1&=\{(x,y)\in C_S\times C_S: y\in R^x_{S\cup {\underline{P^{v,x}_S}}}\},\\
Z_S&=Z_S^0\cup Z_S^1.
\end{align*}
Notice that we have  $Z_S\subseteq (X_0\cup C_S)\times C_S$. %For the rest of the proof $S$ and $T$ will always represent sets of labels with $|S|, |T|\leq k_\lam-3d_\lam$.
For $S$ with $|S|\leq k_\lam-3d_\lam$, (\ref{EqRuRuSDifference}) implies that $|C_S|\geq |R^v|-\Delta_\lam-1\geq |C|-2\delta n$ as required by part (ii) of the lemma. The next two claims prove ``$|Z_S|\geq |(X_0\cup C_S)\times C_S|-4\delta n|C_S|$'', completing the proof of part (ii).
\begin{claim}\label{ZS0DifferenceSmall}
For any $S$ with $|S|\leq  k_\lam-3d_\lam$ we have $|(X_0\times C_S)\setminus Z_S^0|\leq  k_\lam|C_S|$.
\end{claim}
\begin{proof}
We have $(X_0\times C_S)\setminus Z_S^0\subseteq  (S\times C_S)\cup \bigcup_{y\in C_S} \left({\underline {P^{v,y}_S}\times\{y\}}\right)$. Using  $|S|\leq k_\lam-3d_\lam$ and $|\underline {P^{v,y}_S}|\leq 3d_{\lam}$, this implies $|(X_0\times C_S)\setminus Z_S^0|\leq |S||C_S|+\sum_{y\in C_S } |{\underline {P^{v,y}_S}}|\leq k_\lam|C_S|.$
\end{proof}

\begin{claim}\label{ClaimZS1DifferenceSmall}
For any $S$ with $|S|\leq  k_\lam-3d_\lam$ we have $|(C_S\times C_S)\setminus Z_S^1|\leq 2\delta n|C_S|$.
\end{claim}
\begin{proof}
Notice that we have 
\begin{equation}\label{EQZS1Bound}
|Z^1_S|=\sum_{x\in C_S}|R^x_{S\cup {\underline {P^{v,x}_S}}}\cap C_S|=  \sum_{x\in C_S}(|C_S|- |C_S\setminus R^x_{S\cup {\underline {P^{v,x}_S}}}|).
\end{equation}
We will bound the second term by the following
$$|C_S\setminus R^x_{S\cup {\underline {P^{v,x}_S}}}|\leq |C\setminus R^x_{S\cup {\underline {P^{v,x}_S}}}|\leq |C\setminus R^x|+ |R^x\setminus R^x_{S\cup {\underline {P^{v,x}_S}}}|\leq \delta n+\Delta_\lam\leq 2\delta n.$$
The second last inequality comes from $|C\setminus R^x|\leq \delta n$, $|S\cup \underline{P_S^{v,x}}|\leq k_{\lam}$, and (\ref{EqRuRuSDifference}).
Plugging the above into (\ref{EQZS1Bound}) we get $|Z^1_S|\geq \sum_{x\in C_S}(|C_S|-2\delta n)= |C_S|^2-2\delta n|C_S|$ as 
required.
\end{proof}
From Claims~\ref{ZS0DifferenceSmall} and~\ref{ClaimZS1DifferenceSmall}, $n\geq N_0$, and $\lam\leq \lam_\MAX$ we get $|Z_S|\geq |(X_0\cup C_S)\times C_S|-k_\lam |C_S|-2\delta n|C_S|\geq |(X_0\cup C_S)\times C_S|-4\delta n |C_S|$, completing the proof of part (ii) of the lemma.

Next we prove part (i) of the lemma.
\begin{claim}
Let $S$ be a set of $\leq k_\lam-3d_\lam$ labels and $(x,y)\in Z_S$. Then the label $x$ is $(\lam, S)$-amidst $v$ and $y$.
\end{claim}
\begin{proof}
Suppose $x\in X_0$, or equivalently $(x,y)\in Z_S^0$. By definition of $Z_S^0$, we have $x\not\in S$. The path $P^{v,y}_S$ is a length $\leq d_\lam$ switching path from $v$ to $y$ avoiding $S$.  Since $(x,y)\in Z_S^0$, $P^{v,y}_S$ also avoids $x$, and so witnesses $x$ being $(\lam,S)$-amidst $v$ and $y$.

Suppose $x\in C_S$, or equivalently $(x,y)\in Z_S^1$. 
Since $x\in C_S\subseteq R^v_S$, recall that we have a length $\leq d_\lam$ switching path $P^{v,x}_S$ from $v$ to $x$ avoiding $S$. 
Since $v\not \in C_S$, we have $x\neq v$, which implies that $P^{v,x}_S$ has length $\geq 1$.
Since  $y\in R^x_{S\cup {\underline {P^{v,x}_S}}}$, we have  a length $\leq d_\lam$ switching path $P^{x,y}_{S\cup {\underline {P^{v,x}_S}}}$ from $x$ to $y$ avoiding $S$ and $\underline{P^{v,x}_S}$. By Lemma~\ref{LemmaAmidstConcatenate} we have that $Q=P^{v,x}_S + P^{x,y}_{S\cup {\underline {P^{v,x}_S}}}$ is a switching path  witnessing $x$ being $(\lam,S)$ amidst $v$ and $y$.
\end{proof}

The following two claims prove part (iii) of the lemma.
\begin{claim}\label{ClaimZ0SZ0TDifference}
For  $S$ and $T$ with $|S|, |T|\leq  k_\lam-3d_\lam$ we have $|L^+(D)\cap (Z_S^0\setminus Z_T^0)|\leq 2\Delta_\lam n$.
\end{claim}
\begin{proof}
We proceed as follows.
\begin{align*}
|(L^+(D)\cap  (Z_S^0\setminus Z_T^0)|
&\leq |L^+(D)\cap(X_0\times C_S\setminus Z_T^0)|\\
&\leq |L^+(D)\cap(X_0\times C_S\setminus X_0\times C_T)|+|L^+(D)\cap (X_0\times C_T\setminus Z_T^0)|\\
&\leq |L^+(D)\cap(X_0\times C_S\setminus X_0\times C_T)|+ k_\lam|C_S|\\
&= |L^+(D)\cap(X_0\times (C_S\setminus C_T))|+  k_\lam|C_S|\\
&\leq n |C_S\setminus C_T|+ k_\lam |C_S|\\
&\leq \Delta_\lam n+  k_\lam |C_S|\\
&\leq 2\Delta_\lam n.
\end{align*}
The first inequality comes from $Z_S^0\subseteq X_0\times C_S$.
The second inequality is an instance of $U\setminus W\subseteq (U\setminus V) \cup (V\setminus W)$. 
The third inequality comes from Claim~\ref{ZS0DifferenceSmall}.
The equality is an instance of $U\times V\setminus U\times W= U\times (V\setminus W)$.
The fourth inequality comes from $|L^+(D)\cap ((V(D)\cup X_0)\times U)|=  \sum_{u\in U}|N^+(u)|  \leq n|U|$ which holds for any set of vertices $U$. 
The fifth inequality comes from $|R^u_S\setminus R^u_T|\leq \Delta_\lambda.$ 
The sixth inequality holds since $k_\lam\leq \Delta_\lam$ for all $\lam\geq 3$.
\end{proof}

\begin{claim}\label{ClaimZ1SZ1TDifference}
For  $S$ and $T$ with $|S|, |T|\leq  k_\lam-3d_\lam$ we have $|Z_S^1\setminus Z_T^1|\leq 4\Delta_\lam |C_S|$.
\end{claim}
\begin{proof}
Using the definitions of $Z_S$ and $C_S$ we have the following.
\begin{align*}
|Z_S^1\setminus Z_T^1|
&=|((C_S\setminus C_T)\times C_S)\cap (Z_S^1\setminus Z_T^1)| + |((C_S\cap C_T)\times C_S)\cap (Z_S^1\setminus Z_T^1)|\\
&=\sum_{x\in C_S\setminus C_T}|C_S\cap R^x_{S\cup {\underline {P^{v,x}_S}}}|+\sum_{x\in C_S\cap C_T}|(C_S\cap R^x_{S\cup {\underline {P^{v,x}_S}}})\setminus (C_T\cap R^x_{T\cup {\underline {P^{v,x}_T}}})|\\
&\leq \sum_{x\in C_S\setminus C_T}|C_S|+\sum_{x\in C_S\cap C_T}|C_S\setminus C_T|+\sum_{x\in C_S\cap C_T}  |R^x_{S\cup {\underline {P^{v,x}_S}}} \setminus R^x_{T\cup {\underline {P^{v,x}_T}}}|\\
&\leq 4\Delta_\lam|C_S|.
\end{align*}
The first equality comes from $Z_S^1\setminus Z_T^1\subseteq C_S\times C_S$.
The second equality comes from $Z_S^1=\bigcup_{x\in C_S}\{x\}\times (C_S\cap R_{S\cup \underline{P_{S}^{v,x}}}^x)$ and $Z_T^1=\bigcup_{x\in C_T}\{x\}\times (C_T\cap R_{T\cup \underline{P_{T}^{v,x}}}^x)$.
The first inequality comes from $(U\cap W)\setminus (U'\cap W')\subseteq (U\setminus U')\cup (W\setminus W')$.
The second inequality comes from (\ref{EqRuRuSDifference}), $|S\cup \underline{P_S^{v,x}}|\leq k_\lam$, and $|T\cup \underline{P_T^{v,x}}|\leq k_\lam$.
%$|R^u_S\setminus R^u_T|\leq \Delta_\lambda$ holding for any $S$ and $T$ with $|S|, |T|\leq k_\lam$.
\end{proof}
Claims~\ref{ClaimZ0SZ0TDifference} and~\ref{ClaimZ1SZ1TDifference} imply (iii), concluding the proof of the lemma.
\end{proof}

We now prove the main result of this section. The following lemma shows that if we have a digraph $D$ which doesn't satisfy Theorem~\ref{TheoremDigraph}, then $\lam$-components in $D$ have a special property---every vertex in a $\lam$-component $C$ $(\lam+1)$-reaches a set $R$ which is much larger than $C$. This lemma will later be combined with Lemma~\ref{ComponentTower} in order to show that all digraphs satisfy Theorem~\ref{TheoremDigraph}.
\begin{lemma}\label{ComponentGrowth}
Let $D$ be a out-properly labelled, digraph on $n\geq N_0$ vertices, $X_0$ the set of labels of $D$ which are not vertices, and  $3\leq \lam\leq \lam_\MAX$. Let $C$ be a $\lam$-component in $D$, $u\in V(D)$, and $P$ a length $\leq d_{\lam_\MAX}$ switching path from $u$ to some $v\in C$.
Then one of the following holds.
\begin{enumerate}[(i)]
\item There is a vertex $y\in C$, and a set of  labels $A$  amidst $u$ and $y$ such that $|N^+_A(y)|< |A|-|X_0|+\epsilon  n$.
\item $v$ $(\lam+1)$-reaches a set $R$ with 
$|R|\geq |C|+(\epsilon  -7\delta)n.$
\end{enumerate}
\end{lemma}

\begin{proof}
Suppose that (i) doesn't hold. 
\begin{claim}\label{ClaimExpansionAmidstVandY}
Let $y\in C$ and $A$ be a set of labels which are $(\lam, {\underline P})$-amidst $v$ and $y$. Then  $|N^+_A(y)|\geq |A|-|X_0|+\epsilon  n$.
\end{claim}
\begin{proof}
It is sufficient to show that every label in $A$ is amidst $u$ and $y$. Then the claim follows since we are assuming that (i) doesn't hold. 

Fix $a\in A$.
Using the definition of $a$ being $(\lam, {\underline P})$-amidst $v$ and $y$, there is a  switching path $Q$ from $v$ to $y$ avoiding ${\underline P}$ and having no edges labelled by $a$. 
In addition if $a$ is a vertex of $D$, then $a\in Q\setminus\{v\}$. 
We also have $a\not\in \underline P$ since $a$ is $(\lam, \underline P)$-amidst $v$ and $y$.
Since $Q$ avoids ${\underline P}$,  Lemma~\ref{LemmaConcatenatePaths} implies that $P+Q$ is a switching path from $u$ to $y$.
Since neither $P$ nor $Q$ had edges labelled by $a$, $P+Q$ also has no edges labelled by $a$.
In addition, if $a$ is a vertex, then $P+Q$ passes through $a$ and $a\neq u$ (since $a\not\in \underline P$.)
Therefore $P+Q$ witnesses $a$ being amidst $u$ and $y$ as required.
\end{proof}

Notice that every $x\in X_0\setminus \underline{P}$ is $(\lam,\underline{P})$-amidst $v$ and $v$ (witnessed by the single-vertex path $v$.) By Claim~\ref{ClaimExpansionAmidstVandY} applied with $A= X_0\setminus \underline{P}$ we have $|N^+_{X_0\setminus \underline{P}}(y)|\geq \epsilon  n-|\underline P|$. By Lemma~\ref{ComponentLarge}, $n\geq N_0$, and $\underline{P}\leq 3d_{\lam_\MAX}$ we have 
\begin{equation}\label{EqComponentLargeInComponentGrowth}
|C|\geq |N^+_{X_0}(v)|-\delta n-\gamma_{\lam-3} n\geq \epsilon  n-|\underline P|-\delta n-\gamma_{\lam-3} n\geq \epsilon n/2.
\end{equation}

Apply Lemma~\ref{MengerLemma} to $C$ and $v$ to assign sets  $C_S\subseteq C$ and $Z_S\subseteq (X_0\cup C_S)\times C_S$   satisfying all the conclusions of Lemma~\ref{MengerLemma} to every set of labels $S$ with $|S|\leq k_\lam-3d_{\lam}$.
%From Lemma~\ref{MengerLemma}, we have 
%\begin{equation}\label{EqLowerBoundOnZS}
%|Z_S|\geq |C_S|^2+|X_0||C_S|-4\delta|C_S|n .
%\end{equation}
For $z\in V(D)$ and $S$ a set of labels, let 
$$E_S(z)=\{(x,y)\in L^+(D)\cap Z_{S}\text{ with }yz\text{ present and labelled by }x\}.$$ 
Since the labelling on $D$ is out-proper, we have that $E_{S}(z)\cap E_{S}(z')=\emptyset$ for any $S$ and $z\neq z'$.
Also notice that for sets $S$ and $S'$ we have $E_S(z)\cap L^+(D)\cap Z_{S'}\subseteq E_{S'}(z)$.
We now define the set $R$ 
$$R=\{z\in V(D): |E_{{\underline P}}(z)|\geq \delta |C_{\underline{P}}|\}.$$

First we show that $v$ $(\lam+1)$-reaches $R$.
\begin{claim}\label{ClaimRReached}
$v$ $(\lam+1)$-reaches $R$.
\end{claim}
\begin{proof}
Let $S$ be a set of at most $k_{\lam+1}$ labels. Notice that  $|S\cup {\underline P}|\leq k_{\lam+1}+3d_{\lam_\MAX}\leq k_\lam-3d_\lam$.
Let $R_S=\{z\in R: E_{S\cup {\underline P}}(z)\neq \emptyset\}.$

We claim that for every $z\in R_S\setminus (S\cup P)$, there is a length $\leq d_{\lam+1}$ switching path $P_{v,z}$ from $v$ to $z$ avoiding $S$.
Notice that for every $z\in  R_S\setminus (S\cup P)$, we have $E_{S\cup {\underline P}}(z)\neq \emptyset$, and so by the definition of $E_{S\cup {\underline P}}(z)$, there is a pair $(x,y)\in L^+(D)\cap Z_{S\cup {\underline P}}$ with $yz$ present and labelled by $x$. Since $(x,y)\in Z_{S\cup {\underline P}}$, by part (i) of Lemma~\ref{MengerLemma} we have $x$ $(\lam, S\cup {\underline P})$-amidst $v$ and $y$. By Lemma~\ref{LemmaAmidstAddOneVertex} applied with $S'=S\cup \underline P$, there is a switching path from $v$ to $z$ avoiding $S\cup {\underline P}$ of length $\leq 2d_\lam+1\leq d_{\lam+1}$.

To prove the claim it is sufficient  to show that $|R\setminus (R_S\setminus (S\cup P))|\leq \Delta_{\lam+1}$. We will do this by showing $|R\setminus R_S|+|S\cup P|\leq \Delta_{\lam+1}$.
We have
\begin{align*}
 \delta |C_{\underline P}||R\setminus R_S|
					&\leq \sum_{z\in R\setminus R_S}|E_{{\underline P}}(z)|\\
					&=\left|\bigcup_{z\in R\setminus R_S}E_{{\underline P}}(z)\right|\\
					&\leq |(L^+(D)\cap Z_{{\underline P}})\setminus (L^+(D)\cap Z_{S\cup {\underline P}})|\\
					&\leq 6\Delta_{\lam} n.
\end{align*}
The first inequality comes from the definition of $R$. The equality comes from $E_{{\underline P}}(z)\cap E_{{\underline P}}(z')=\emptyset$ for $z\neq z'$. 
For the second inequality  first recall that $E_{\underline P}\subseteq L^+(D)\cap Z_{\underline{P}}$. Then notice that for $z\in R\setminus R_S$ we have $E_{S\cup \underline P}(z)=\emptyset$ and hence $E_{\underline{P}}(z)\cap L^+(D)\cap Z_{S\cup \underline P}\subseteq E_{S\cup \underline P}(z)=\emptyset$. This shows  $E_{{\underline P}}(z) \subseteq (L^+(D)\cap Z_{{\underline P}})\setminus (L^+(D)\cap Z_{S\cup {\underline P}})$ for $z \in R\setminus R_S$ which implies the second inequality.
The third inequality comes from part (iii) of Lemma~\ref{MengerLemma} and $|S\cup {\underline P}|\leq k_\lam-3d_\lam$.

Rearranging and using $|S\cup P|\leq k_{\lam}$ we obtain $|R\setminus R_S|+|S\cup P|\leq 6\Delta_{\lam}n/\delta |C_{\underline{P}}|+ k_{\lam}$. From (\ref{EqComponentLargeInComponentGrowth}) and Lemma~\ref{MengerLemma} (ii) we have $|C_P|\geq |C|-\delta n\geq \epsilon n/3$.
Combining these  gives $|R\setminus R_S|+|S\cup P|\leq 18\Delta_{\lam}/\delta \epsilon+ k_{\lam}\leq \Delta_{\lam + 1}$ as required.
\end{proof}

%To prove the lemma we need to show that $R$ is large. Recall that $R$ is defined in terms of the sets $E_S^L(z)$ which are defined  in terms of $L^+(D)\cap Z_{S}$. Because of this showing that $R$ is large basically comes down to showing that the set $L^+(D)\cap Z_{S}$ is large. Unfortunately we currently only have a lower bound on $Z_{S}$ rather than $L^+(D)\cap Z_{S}$ (coming from Lemma~\ref{MengerLemma} (ii)). 
The following claim  lets us lower bound $|L^+(D)\cap Z_{\underline{P}}|$ in terms of $|Z_{\underline{P}}|$.
\begin{claim}\label{EQBoundLZP}
$|L^+(D)\cap Z_{\underline P}|\geq |Z_{\underline P}|-|X_0||C_{\underline P}|+\epsilon n|C_{\underline P}|.$
\end{claim}
\begin{proof}
For every $y\in C_{\underline P}$, define
\begin{align*}
A_{\underline P}(y)&=\{x\in X_0\cup C_{\underline P}: (x,y)\in Z_{\underline P}\},\\
A_{\underline P}^L(y)&=\{x\in X_0\cup C_{\underline P}: (x,y)\in L^+(D)\cap Z_{\underline P}\}.
\end{align*}
From the definition of $L^+(D)$ we have $A_{\underline P}^L(y)=\{x\in A_{\underline P}(y): |N^+_{\{x\}}(y)|=1\}$.
%Notice that for $y\neq y'$ we have $A_{\underline P}(y)\cap A_{\underline P}(y)=\emptyset$ and $A^L_{\underline P}(y)\cap A^L_{\underline P}(y)=\emptyset$.
Notice that we have $Z_{\underline P}=\bigcup_{y\in C_{\underline{P}}} A_{\underline{P}}(y)\times\{y\}$ and $L^+(D)\cap Z_{\underline P}=\bigcup_{y\in C_{\underline{P}}} A^L_{\underline{P}}(y)\times\{y\}$.
The following string of equalities holds.
\begin{equation}\label{EqBoundLZP1}
|L^+(D)\cap Z_{\underline P}|=\sum_{y\in C_{\underline P}}|A_{\underline P}^L(y)|= \sum_{y\in C_{\underline P}}\big|N_{A_{\underline P}^L(y)}^+(y)\big|= \sum_{y\in C_{\underline P}}\big|N_{A_{\underline P}(y)}^+(y)\big|.
\end{equation}
The first equality comes from $L^+(D)\cap Z_P=\bigcup_{y\in C_{\underline{P}}} A^L_{\underline{P}}(y)\times\{y\}$.
The second equality holds since $\big|N^+_{\{x\}}(y)\big|=1$ for $x\in A_{\underline P}^L(y)$ which implies $\big|N^+_{A_{\underline P}^L(y)}(y)\big|=|A_{\underline P}^L(y)|$. 
The third equality holds since $\big|N_{A_{\underline P}(y)}^+(y)\big|=\big|N_{A_{\underline P}^L(y)}^+(y)\big|+\big|N_{A_{\underline P}\setminus A_{\underline P}^L(y)}^+(y)\big|$ and also $\big|N_{A_{\underline P}\setminus A_{\underline P}^L(y)}^+(y)\big|=0$ which comes from ``$A_{\underline P}^L(y)=\{x\in A_{\underline P}(y): |N^+_{\{x\}}(y)|=1\}$''.

By Lemma~\ref{MengerLemma} (i), for every $x\in A_{\underline P}(y)$ we have $x$ $(\lam, \underline{P})$-amidst $v$ and $y$.
We can use Claim~\ref{ClaimExpansionAmidstVandY} to bound $\sum_{y\in C_{\underline P}}\left|N_{A_{\underline P}(y)}^+(y)\right|$.
\begin{equation}\label{EqBoundLZP2}
\sum_{y\in C_{\underline P}}\left|N_{A_{\underline P}(y)}^+(y)\right|
				\geq \sum_{y\in C_{\underline P}}(|A_{\underline P}(y)|-|X_0|+\epsilon n)
				= |Z_{\underline P}|-|X_0||C_{\underline P}|+\epsilon n|C_{\underline P}|
\end{equation}
 The first inequality comes from Claim~\ref{ClaimExpansionAmidstVandY}. The equality comes from $Z_P=\bigcup_{y\in C_{\underline{P}}} A_{\underline{P}}(y)\times\{y\}$. The claim  follows from~(\ref{EqBoundLZP1}) and~(\ref{EqBoundLZP2}).
% The second inequality comes from (\ref{EqLowerBoundOnZS}). 
\end{proof}

Now we show that $R$ is large.
\begin{claim}\label{ClaimRLarge}
$|R|\geq |C|+(\epsilon -7\delta)n.$
\end{claim}
\begin{proof}
Since $D$ doesn't have repeated edges, for any $z\in V(D)$ and $y\in C_{\underline P}$, there can be at most one label $x$  for which $(x,y)\in E_{{\underline P}}(z)$. In particular this implies that for all $z\in V(D)$ we have $|E_{{\underline P}}(z)|\leq |C_{\underline P}|$.
For $z\not\in R$ we have $|E_{{\underline P}}(z)|\leq  \delta |C_{\underline P}|$. 
These imply
\begin{equation}\label{EqABoundClaim1}
(n-|R|)\delta |C_{\underline{P}}|+|R||C_{\underline P}|\geq \sum_{z\in V(D)}|E_{{\underline P}}(z)|=|L^+(D)\cap Z_{\underline P}|.
\end{equation}
To see the equality, notice that by the definition of $E_{\underline P}(z)$ both sides equal $\left|\bigcup_{z\in V(D)}E_{\underline P}(z)\right|$.
Combining Claim~\ref{EQBoundLZP} with part (ii) of Lemma~\ref{MengerLemma} we get 
\begin{equation}\label{EqABoundClaim2}
|L^+(D)\cap Z_{\underline P}|\geq |Z_{\underline P}|-|X_0||C_{\underline P}|+\epsilon n|C_{\underline P}|\geq |C_{\underline P}|^2+\epsilon n|C_{\underline P}|-4\delta n|C_{\underline P}|.
\end{equation}
Combining (\ref{EqABoundClaim1}) and (\ref{EqABoundClaim2}) and rearranging implies the claim. 
\begin{align*}
|R|\geq\frac{|C_{\underline P}|^2+\epsilon n|C_{\underline P}|-5\delta n|C_{\underline P}|}{|C_{\underline P}|-\delta |C_{\underline{P}}|}
\geq |C_{\underline P}|+\epsilon  n-5\delta n.
\geq |C|+\epsilon  n-7\delta n.
\end{align*}
The last inequality is from Lemma~\ref{MengerLemma} (ii).
\end{proof}
Claims~\ref{ClaimRReached} and~\ref{ClaimRLarge} prove the lemma.
\end{proof}

\subsection{Proofs of Theorems~\ref{MainTheorem} and~\ref{TheoremDigraph}}
In this section we prove the theorems of this paper. 
\begin{proof}[Proof of Theorem~\ref{TheoremDigraph}]
Suppose that there is a vertex $u$ for which the theorem doesn't hold i.e. that for every vertex $v$ and set of labels $A$ amidst $u$ and $v$  we have $|N^+_A(v)|\geq |A|-|X_0|+\epsilon  n$.
Apply Lemma~\ref{ComponentTower} to $u$  in order to obtain numbers $\lam_1, \dots, \lam_{4\delta^{-1}}$, components $C_1, \dots, C_{4\delta^{-1}}$,   and a path $P$ passing through all of them. For each $i$, let $v_i$ be a vertex in $P\cap C_i$.

Using Lemma~\ref{ComponentLarge} we can show that all the components $C_i$ are large.
\begin{claim}\label{ClaimComponentBiggerThanX0}
$|C_i|\geq \epsilon n/2$ for $i=1, \dots, 4\delta^{-1}$.
\end{claim}
\begin{proof}
Fix $i\leq 4\delta^{-1}$. Let $P'$ be the subpath of $P$ ending with $v_i$ and $A=X_0\setminus \underline{P'}$. Notice that $P'$ witnesses every label in $A$ being amidst $u$  and $v_i$. Since we are assuming that the theorem doesn't hold we obtain 
$|N^+_A(v_i)|\geq |A|-|X_0|+\epsilon  n\geq \epsilon n-|\underline {P'}|\geq \epsilon  n -3d_{\lam_\MAX}$ (using the fact that $P$ has length $\leq d_{\lam_\MAX/2}$ in Lemma~\ref{ComponentTower}.)
Since $N^+_{A}(v_i)\subseteq N^+_{X_0}(v_i)$, $n\geq N_0$, and $\lam_i\leq \lam_\MAX$, Lemma~\ref{ComponentLarge}  implies $|C_i|\geq |N^+_{A}(v_i)|-\delta n-\gamma_{\lam_i-3}n\geq \epsilon  n -3d_{\lam_\MAX} -\delta n-\gamma_{\lam_i-3}n\geq \epsilon n/2$.
\end{proof}

Using Lemma~\ref{ComponentGrowth}, we can show that each $C_i$ is much bigger than the previous one.
\begin{claim}\label{ClaimComponentGrowth}
$|C_{i+1}|\geq |C_i|+\epsilon  n/2$ for $i=1, \dots, 4\delta^{-1}$
\end{claim}
\begin{proof}
Recall that from Lemma~\ref{ComponentTower} we have $|C_i\setminus C_{i+1}|\leq \delta n$ for all $i$. Combining this with Claim~\ref{ClaimComponentBiggerThanX0} we get $|C_i\cap C_{i+1}|=|C_i|-|C_i\setminus C_{i+1}|\geq \epsilon n/2-\delta n\geq \epsilon  n/3$. For each $i$, let $R_{v_i}$ be a set from the definition $\lam_i$-component which is  $\lam_i$-reached by $v_i$. 
Using $|C_i\symdiff R_{v_i}|\leq \delta n$ and $|C_i\cap C_{i+1}|\geq  \epsilon  n/3$ we get $|R_{v_i}\cap C_{i}\cap C_{i+1}|\geq |C_{i}\cap C_{i+1}|- |C_i\setminus R_{v_i}|\geq  \epsilon  n/3-\delta n\geq \Delta_{\lam_i}+1$. From the monotonicity of reaching we have that $v_i$ $\lam_i$-reaches $R_{v_i}\cap C_{i}\cap C_{i+1}$.  Using $|\underline P|\leq 3d_{\lam_\MAX/2}\leq k_{\lam_i}$ and $|R_{v_i}\cap C_{i}\cap C_{i+1}|\geq \Delta_{\lam_i}+1$ this implies that there is a length $\leq d_{\lam_i}$ switching path $Q$  avoiding $\underline P$ from $v_i$ to a vertex $v\in R_{v_i}\cap C_{i}\cap C_{i+1}$. 
Let $P'$ be the subpath of $P$ from $u$ to $v_i$. By Lemma~\ref{LemmaConcatenatePaths}, $P'+Q$ is a length $\leq d_{\lam_\MAX}$ switching path from $u$ to $v$.

Notice that all the conditions of Lemma~\ref{ComponentGrowth} hold with $C=C_i$, $\lam=\lam_i$, $u=u$, $v=v$, and $P=P'+Q$. In addition, (i) cannot hold since we are assuming that $u$  is a vertex for which the theorem is false.
Therefore part (ii) of Lemma~\ref{ComponentGrowth} occurs, i.e. we obtain a set $R$ which is $(\lam_i+1)$-reached by $v$ and has $|R|\geq |C_i|+(\epsilon -7\delta)n$. Since $v$ is in the $(\lam_{i+1})$-component $C_{i+1}$, we have a set $R_{v}$ with $|R_v\symdiff C_{i+1}|\leq \delta n$ such that $\overline{R_v}$ is $(\lam_{i+1}-3)$-bypassed by $v$. 
From Lemma~\ref{ComponentTower} we have $\lam_i+1\leq \lam_{i+1}-3$.
Since $v$ $(\lam_i+1)$-reaches $R$, $(\lam_{i+1}-3)$-bypasses $\overline{R_v}$, and $\lam_i+1\leq \lam_{i+1}-3$ we get $|R\setminus {R_v}|=|R\cap \overline{R_v}|\leq \gamma_{(\lam_{i+1}-3)}n$ (using the monotonicity of reaching.) Combining this with $|R_v\symdiff C_{i+1}|\leq \delta n$, $|R|\geq |C_i|+(\epsilon -7\delta)n$, and $n\geq N_0$ we get 
\begin{align*}
|C_{i+1}|&\geq |R_v|-|R_v\setminus C_{i+1}|
\geq |R|-|R\setminus R_v|-|R_v\symdiff C_{i+1}|\\
&\geq |C_i|+ (\epsilon  -7\delta)n -\gamma_{(\lam_{i+1}-3)}n-\delta n
\geq |C_i|+\epsilon  n/2.
\end{align*}
\end{proof}

Iterating $|C_{i+1}|\geq |C_i|+\epsilon  n/2$ for $i=1, \dots, 4\delta^{-1}$ gives $|C_{4\delta^{-1}}|\geq 2\delta^{-1}\epsilon n>n$ which is a contradiction to $C_{4\delta^{-1}}\subseteq V(D)$.
\end{proof}

Using Theorem~\ref{TheoremDigraph} and Lemma~\ref{LemmaMaximumMatchingExpansion} it is easy to prove our approximate version of Conjecture~\ref{AharoniBergerConjecture}.
\begin{proof}[Proof of Theorem~\ref{MainTheorem}]
%By deleting some edges from the graph, without loss of generality, we can suppose that $\epsilon<0.9$.
Let $N_0=N_0(0.9\epsilon)$ be the constant from Theorem~\ref{TheoremDigraph}, and let $G$ be a properly coloured bipartite multigraph with $n\geq N_0$ colours having $\geq (1+\epsilon)n$ edges of each colour.
%Without loss of generality we can suppose that there are exactly $(1+\epsilon )n$ edges in each colour.
Let $M$ be a maximum rainbow matching in $G$. 
Let $X$ and $Y$ be the parts of the bipartition of $G$ and $X_0=X\setminus V(M)$.
Suppose for the sake of contradiction that $M$ misses a colour $c^*$. Let $\DGM$ be the labelled directed graph from Definition~\ref{DefinitionDGM} corresponding to $M$. By Lemma~\ref{LemmaProperlyColoured}, $\DGM$ is out-properly labelled and simple.
By Lemma~\ref{LemmaMaximumMatchingExpansion} we have that for any vertex $v\in V(\DGM)$ and any set of labels $A$ amidst $c^*$ and $v$ we have $|N^+_A(v)|\geq |A|-|X_0|+\epsilon  n-1\geq |A|-|X_0|+0.9\epsilon n$. This contradicts Theorem~\ref{TheoremDigraph} applied with $\epsilon'=0.9\epsilon$.
\end{proof}

\section{Concluding remarks}
Here we make some remarks about the proof in this paper and directions for further research.

\subsubsection*{Improving the bound in Theorem~\ref{MainTheorem}}
In this paper we proved an approximate version of the Aharoni-Berger Conjecture. Naturally, the main direction for further research is to improve the dependency of $N_0$ on $\epsilon$ in Theorem~\ref{MainTheorem}, and eventually prove the full conjecture. The dependency which follows from our proof is extremely bad---for $\epsilon>0$, we have $N_0=\twr\left(2\cdot 4^{\epsilon^{-9}}\right)$. This dependency can surely be significantly improved by tweaking the proof in various ways. The author believes that getting a polynomial error term is out of reach of the methods in this paper.
\begin{problem}\label{Problem_Polynomial_Bounds}
For some $\alpha<1$ prove the following.
 Let $G$  be a properly edge-coloured bipartite multigraph with $n$ colours and at least $n+n^{\alpha}$ edges of each colour. Then $G$ has a  rainbow matching using every colour.
\end{problem}
Of particular interest would be to solve the above problem for some $\alpha<1/2$. This is because there are some natural variants of the Aharoni-Berger Conjecture, where $n^{1/2}$ is the best currently known bound on the error term. One of these is the version of the Aharoni-Berger Conjecture where not every colour needs to be used in the rainbow matching. Specifically it is know that in every properly edge-coloured bipartite multigraph with $n$ colours and at least $n$ edges of each colour, there is a rainbow matching of size $n-\sqrt n$ (see \cite{Barat, Woolbright}.) 
Also recall that Haggkvist and Johansson proved an approximate version of the Aharoni-Berger Conjecture when when the colour classes in $G$ are all disjoint perfect matchings. In their paper \cite{HaggkvistJohansson} they say that ``it will be clear from the proof that in  order  to  reach $\epsilon\leq n^{-1/2}$ some  new  ideas  must  be  found,  if  indeed  the theorem is valid in this range.'' This suggests that $\sqrt n$ is a natural barrier for their techniques as well. 

Finally it would be extremely interesting to show that every properly $n$-edge-coloured bipartite multigraph with $n+o(\log^2 n)$ edges of each colour has a rainbow matching using every colour. This would improve the best known bound on the Brualdi-Stein Conjecture \cite{HatamiSchor}.

After this paper was announced, Gao, Ramadurai, Wanless, and Wormald~\cite{gao2017} solved Problem~\ref{Problem_Polynomial_Bounds} when an additional restriction is placed on the multigraph --- that the multiplicity of each edge is at most $\sqrt{n}/\log^2 n$. Together with the Haggkvist-Johanson approach~\cite{HaggkvistJohansson}, and the author's first approach~\cite{PokrovskiyRainbow1}, this gives three different approaches which can prove an asymptotic version of the Aharoni-Berger Conjecture when the underlying graph is simple. However, when the underlying graph is a general multigraph, then the approach in this paper remains the only known way to prove the Aharoni-Berger Conjecture asymptotically.

\subsubsection*{Improving the bound in Theorem~\ref{TheoremDigraph}}
In this paper and~\cite{PokrovskiyRainbow1} we introduced a directed graph based approach to the Aharoni-Berger Conjecture. It would be interesting to know how far this approach can be pushed. A specific open problem is to find out how small $N_0$ in Theorem~\ref{TheoremDigraph} can be. 
Perhaps with some completely different proof technique, Theorem~\ref{TheoremDigraph} can be proved with much better bounds?
Of particular interest is to find out whether there are serious barriers to this approach proving the full Aharoni-Berger Conjecture. For example---for every $C\in \mathbb{N}$, are there labelled directed graphs such that for all $u, v\in V(D)$ and  set of labels $A$ amidst $u$ and $v$ we have $|N^+_A(v)|\geq |A|-|X_0|+C$?''

\subsubsection*{Non-bipartite graphs}
The problem dealt with in this paper can be asked for non-bipartite graphs as well i.e. what is the smallest $f(n)$ so that given $n$ matchings of size $f(n)$  in a (not necessarily bipartite) graph $G$, there is a rainbow matching using every colour. The following conjecture about this appears in \cite{gao2017}.
\begin{conjecture}[\cite{gao2017}]\label{Conjecture_Nonbipartite}
 Let $G$  be a properly edge-coloured bipartite multigraph with $n$ colours having at least $n+2$ edges of each colour. Then $G$ has a  rainbow matching using every colour. 
\end{conjecture}
The motivation for asking for $n+2$ edges of each colour (rather than $n+1$ like in Conjecture~\ref{AharoniBergerConjecture}) is that in the non-bipartite case there is an example of graph with $n+1$ edges of each colour and no rainbow matching using each colour. This example is to take two vertex-disjoint copies of $K_4$ and properly edge-colour it using $3$ colours. This graph has $4$ edges of each colour, but it is easy to check that it has no rainbow matching of size $3$.

Some progress has been made on Conjecture~\ref{Conjecture_Nonbipartite} since this paper has been announced. 
Keevash and Yepremyan~\cite{keevash2017} showed that if the multiplicities are $\leq o(n)$ and each colour appears at least $(1+o(1))n$ times, then there is a rainbow matching using $n-O(1)$ colours. 
Gao, Ramadurai, Wanless, and Wormald~\cite{gao2017} proved that if the multiplicities are $\leq \sqrt{n}/\log^2 n$ and each colour appears at least $(1+o(1))n$ times, then there is a rainbow matching using every colour. 
When restricted to bipartite graphs, both of these results are qualitatively weaker than the one in this paper (since Theorem~\ref{MainTheorem}  places no restriction on the multiplicities). 

It is natural to ask whether the methods in this paper can prove an approximate version Conjecture~\ref{Conjecture_Nonbipartite} without any restriction on the multiplicities. While the  author isn't aware of any inherent barriers preventing a suitable generalization from existing, it certainly is no easy task to find one. Indeed 
the construction of the auxiliary directed graph $\DGM$ relied heavily on the graph $G$ being bipartite, and it is not clear what natural auxiliary directed graph could be useful in the non-bipartite case.

\subsubsection*{Different approaches based on directed graphs}
Theorem~\ref{MainTheorem} was proved by considering a directed graph $\DGM$ corresponding to $G$ and studying paths called ``switching paths'' in $\DGM$. Neither the definition of $\DGM$ nor the notion of ``switching path'' which we used are canonical. There are variations of these definitions which could be used to prove the same theorems.
For example, instead of the directed graph $\DGM$ perhaps one could consider  an alternative definition where edges going through $Y_0$ in $G$ somehow corresponded to edges in $\DGM$. 
Instead of switching paths, one can use the following. 
\begin{definition}[Weakly switching path]
A path $P=(p_0, \dots, p_d)$ in an edge-labelled, directed graph $D$ is a weakly switching path if the following hold.
\begin{itemize}
\item $P$ is rainbow  i.e. the edges of $P$ have different labels.
\item If $p_{i}p_{i+1}$ is labelled by a vertex $v\in V(D)$, then $v=p_j$ for some $1\leq j\leq d$.
\end{itemize}
\end{definition}
The difference between ``weakly switching path'' and ``switching path'' is that for weakly switching paths if $p_{i}p_{i+1}$ is labelled by $v\in V(D)$ then we only ask for $v$ to be a non-starting vertex of $P$ (whereas for ``switching path'', we wanted $v$ to precede $p_ip_{i+1}$ as well.)
It is not hard to check that everything in Section~\ref{SectionDirectedGraphs} stays true if we replace ``switching path'' by ``weakly switching path''. Also a ``weakly switching path'' version of Theorem~\ref{TheoremDigraph} follows from the version of Theorem~\ref{TheoremDigraph} which we prove  (just because switching paths are a special case of weakly switching paths.) 

The above discussion suggests  that weakly switching paths are perhaps a better notion to use in future research. 
The reason we didn't use them in this paper is a little bit technical. If we changed ``switching path'' to ``weakly switching path'' in the definition of ``$v$ $\lam$-reaching $R$'', then we would allow paths $P$ from $v$ to $x\in R$ which have an edge labelled by $x$. This causes a problem in the proof of Lemma~\ref{MengerLemma} since there we want to construct paths from $v$ through a vertex $r$ with no edges labelled by $x$. However it is not hard to overcome this issue and prove a version of Theorem~\ref{TheoremDigraph} whilst working directly with weakly switching paths (for example by suitably changing the definition of ``reaching''.)

\subsection*{Acknowledgment}
The author would like to thank Janos Barat for introducing him to this problem as well as Ron Aharoni, Dennis Clemens, Julia Ehrenm\"uller, Peter Keevash, and  Tibor Szab\'o for various discussions related to it. Additionally, the author would like to thank a careful referee for many extremely helpful suggestions.

\bibliography{Rainbow}

\begin{thebibliography}{10}
\expandafter\ifx\csname url\endcsname\relax
  \def\url#1{\texttt{#1}}\fi
\expandafter\ifx\csname urlprefix\endcsname\relax\def\urlprefix{URL }\fi
\expandafter\ifx\csname href\endcsname\relax
  \def\href#1#2{#2} \def\path#1{#1}\fi

\bibitem{WanlessSurvey}
I.~M. Wanless, Transversals in {L}atin squares: A survey, in: R.~Chapman (Ed.),
  Surveys in Combinatorics 2011, Cambridge University Press, 2011.

\bibitem{Ryser}
H.~Ryser, Neuere probleme der kombinatorik, Vortr\"age \"uber Kombinatorik,
  Oberwolfach (1967) 69--91.

\bibitem{Brualdi}
R.~A. Brualdi, H.~J. Ryser, Combinatorial matrix theory, Cambridge University
  Press, 1991.

\bibitem{Stein}
S.~K. Stein, Transversals of {L}atin squares and their generalizations, Pacific
  J. Math. 59 (1975) 567--575.

\bibitem{AharoniBerger}
R.~Aharoni, E.~Berger, Rainbow matchings in r-partite r-graphs, Electron. J.
  Combin. 16.

\bibitem{AharoniCharbitHoward}
R.~Aharoni, P.~Charbit, D.~Howard, On a generalization of the
  {R}yser-{B}rualdi-{S}tein conjecture, J. Graph Theory 78 (2015) 143--156.

\bibitem{KotlarZiv}
D.~Kotlar, R.~Ziv, Large matchings in bipartite graphs have a rainbow matching,
  European J. Combin. 38 (2014) 97--101.

\bibitem{PokrovskiyRainbow1}
A.~Pokrovskiy, Rainbow matchings and rainbow connectedness, The Electronic
  Journal of Combinatorics 24~(1).

\bibitem{DennisJulia}
D.~Clemens, J.~Ehrenm\"uller, An improved bound on the sizes of matchings
  guaranteeing a rainbow matching, Electron. J. Combin. 23.

\bibitem{aharoni2017representation}
R.~Aharoni, D.~Kotlar, R.~Ziv, Representation of large matchings in bipartite
  graphs, SIAM Journal on Discrete Mathematics 31~(3) (2017) 1726--1731.

\bibitem{HaggkvistJohansson}
R.~H\"aggkvist, A.~Johansson, Orthogonal {L}atin rectangles, Combin. Probab.
  Comput. 17 (2008) 519--536.

\bibitem{BollobasModernGraphTheory}
B.~Bollob{\'a}s, Modern graph theory, Springer Science \& Business Media, 2013.

\bibitem{Barat}
J.~Bar{\'a}t, A.~Gy{\'a}rf{\'a}s, G.~N. S{\'a}rk{\"o}zy, Rainbow matchings in
  bipartite multigraphs, Periodica Mathematica Hungarica 74~(1) (2017)
  108--111.

\bibitem{Woolbright}
D.~E. Woolbright, An $n\times n$ {L}atin square has a transversal with at least
  $n-\sqrt n$ distinct symbols, J. Combin. Theory Ser. A 24 (1978) 235--237.

\bibitem{HatamiSchor}
P.~Hatami, P.~W. Shor, A lower bound for the length of a partial transversal in
  a {L}atin square, J. Combin. Theory Ser. A 115 (2008) 1103--1113.

\bibitem{gao2017}
P.~Gao, R.~Ramadurai, I.~Wanless, N.~Wormald, Full rainbow matchings in graphs
  and hypergraphs, arXiv:1709.02665.

\bibitem{keevash2017}
P.~Keevash, L.~Yepremyan, Rainbow matchings in properly-coloured multigraphs,
  arXiv:1710.03041.

\end{thebibliography}
\bibliographystyle{abbrv}
\end{document}